\documentclass[11pt,reqno]{amsart}
\usepackage{amsfonts}
\usepackage{amssymb}
\usepackage{textcomp}
\usepackage{graphicx}
\def\sqr#1#2{{\vcenter{\vbox{\hrule height.#2pt
        \hbox{\vrule width.#2pt height#1pt \kern#2pt
        \vrule width.#2pt}
        \hrule height.#2pt}}}}

\newcommand{\nc}{\newcommand}
\nc{\parent}[1]{$[\![#1]\!]$}
\newtheorem{theorem}{Theorem}[section]
\newtheorem{lemma}{Lemma}[section]
\newtheorem{example}{Example}[section]
\newtheorem{corollary}{Corollary}[section]
\newtheorem{proposition}{Proposition}[section]
\newtheorem{remark}{Remark}
\newtheorem{definition}{Definition}[section]
\newtheorem{assumption}{Assumption}[section]

\newenvironment{pf-main}{{\sc Proof of Theorem \ref{mainresult}.}\hspace{3mm}}{\qed}
\nc{\eid}{\stackrel{d}{=}}
\nc{\cadlag}{c\`{a}dl\`{a}g } \nc{\ba}{\begin{array}}
\nc{\ea}{\end{array}} \nc{\be}{\begin{equation}}
\nc{\ee}{\end{equation}} \nc{\bea}{\begin{eqnarray}}
\nc{\eea}{\end{eqnarray}} \nc{\bean}{\begin{eqnarray*}}
\nc{\eean}{\end{eqnarray*}} \nc{\bu}{\bullet} \nc{\nn}{\nonumber}
\nc{\cA}{{\mathcal A}} \nc{\cB}{{\mathcal B}} \nc{\cC}{{\mathcal
C}} \nc{\cD}{{\mathcal D}} \nc{\bbD}{\mathbb{D}}
\nc{\cG}{{\mathcal G}} \nc{\cF}{{\mathcal F}} \nc{\cS}{{\mathcal
S}} \nc{\cU}{{\mathcal U}} \nc{\cH}{{\mathcal H}}
\nc{\cK}{{\mathcal K}}\nc{\cL}{{\mathcal L}} \nc{\cW}{{\mathcal W}}  \nc{\cM}{{\mathcal
M}}\nc{\cN}{{\mathcal N}} \nc{\cO}{{\mathcal O}} \nc{\cP}{{\mathcal P}}
\nc{\bbE}{\mathbb{E}} \nc{\bbW}{\mathbb{W}}\nc{\bbF}{\mathbb{F}}
\nc{\bbEQ}{\mathbb{E}_{\mathbb{Q}}} \nc{\eps}{\varepsilon}
\nc{\bbEP}{\mathbb{E}_{\mathbb{P}}}\nc{\bbL}{\mathbb{L}}
\nc{\what}{\widehat} \nc{\bbP}{\mathbb{P}} \nc{\bbQ}{\mathbb{Q}}
\nc{\del}{\partial} \nc{\Om}{\Omega} \nc{\om}{\omega}
\nc{\bbR}{\mathbb{R}} \nc{\bbN}{\mathbb{N}} \nc{\fps}{$(\Om, \cF,
(\cF_t)_{t\geq 0}, \bbP)$} \nc{\bbC}{\mathbb{C}}
\nc{\bfr}{\begin{flushright}} \nc{\efr}{\end{flushright}}
\nc{\dXt}{\Delta X_{t}} \nc{\dXs}{\Delta X_{s}}
\nc{\bs}{\blacksquare} \nc{\dX}{\Delta X} \nc{\dY}{\Delta Y}
\nc{\dnkx}{\left(X(T^{n}_{k})-X(T^{n}_{k-1})\right)}
\nc{\esssup}{\mathrm{ess}\mbox{ }\mathrm{sup}}
\nc{\essinf}{\mathrm{ess}\mbox{ } \mathrm{inf}}
\nc{\dhats}{\widehat{\delta_s}} \nc{\half} {\frac{1}{2}}
\nc{\ch}{\hat{c}}

\nc{\ol}{\overline}
\def\rar{\rightarrow}

\nc{\chf}{\mbox{$\mathbf1$}}
\begin{document}

\title{Insider trading with penalties, entropy and quadratic BSDEs}
\author{Umut \c{C}etin}
\address{Department of Statistics, London School of Economics and Political Science, 10 Houghton st, London, WC2A 2AE, UK}
\date{\today}
\begin{abstract}
Kyle model in continuous time where the insider may be subject to legal penalties is considered. In equilibrium the insider internalises this legal risk by trading less aggressively. The equilibrium is characterised via the solution of a backward stochastic differential equation (BSDE) whose terminal condition is determined as the fixed point of a non-linear operator in equilibrium. The insider's expected penalties in equilibrium is non-monotone in the fee structure and is given by the relative entropy of the law of a particular $h$-transformation of Brownian motion. 
\end{abstract}
\maketitle

\section{Introduction}
Kyle \cite{Kyle} (together with its formulation in continuous time by Back \cite{Back92})  is the canonical model for analysing the dissemination of private information into  prices in financial markets.  In a market consisting of noise traders, competitive market makers and an informed trader (insider), it predicts that the insider gradually reveals all her private information to  the market in equilibrium. At the end of the finite horizon, the prices become fully informative in the sense that there is no remaining uncertainty about the private valuation of the insider. 

The Kyle model does not distinguish between legal informed trading, i.e. the investor paying a fee to obtain private information typically as a result of fundamental research, and insider trading when the trader receives private information from an insider via illegal means. As there is no legal penalties in the Kyle model, both types of informed traders behave identically. 

 Incorporation of legal risk into the Kyle model is a relatively less studied extension. The previous works in the literature typically assume a one period setting (see, e.g., Shin \cite{ShinIns}  and Carr\'e et al. \cite{CCDG}) with the exception of Kacperczyk and Pagnotta \cite{KPins} studying a two-period model\footnote{See also DeMarzo et al. \cite{DeMarzo} for another study of regulation of insider trading in a one-period model but outside Kyle's framework.}. Although some insiders trade in blocks, they typically split their orders over a relatively long trading period. Indeed, Kacperczyk and Pagnotta \cite{KPins} find that ``the median trader splits trades over a period equivalent to 71\% of the information horizon.'' Thus, understanding the impact of legal risk in continuous time is essential for the regulation of illegal insider trading.
 
  This paper analyses the impact of legal penalties on insider trading in a continuous time Kyle model.  Another notable distinction from the earlier works is  the fundamental value $V$ of the traded asset is allowed to have an arbitrary distribution while the aforementioned works assume a Gaussian or uniform distribution. Exi

Following \cite{ShinIns} and \cite{KPins} the insider pays  a quadratic penalty if caught by the authorities as a result of a successful investigation, in addition to losing all her previous gains from trading. The presence of legal penalties makes the resolution of equilibrium in closed form less likely. However, quadratic (or more general convex) penalties results in the optimal strategy of the insider being given in the feedback form as a function of the pricing rule to be determined in equilibrium. Section \ref{s:formal} shows that the value function of the insider is the sum of two components: while the first component is a solution of heat equation reminiscent of the value function of the Kyle model with no penalties, the second component is given by the solution of the following  quadratic BSDE (backward stochastic differential equation)
\be \label{e:Uintro}
dU_t=\sigma Z_t dB_t -\frac{1}{2c}Z_t^2 dt,
\ee
whose terminal condition is to be determined in equilibrium. Moreover, the optimal strategy of the insider coincides with a constant multiple of the process $Z$ appearing in the above BSDE, whose explicit solution is available once its terminal condition is established. 

Section \ref{s:fp} describes an operator whose fixed point, if it exists,  yields the terminal condition of the BSDE \eqref{e:Uintro}. This requires finding two convex functions $\phi$ and $\Psi$ such that
\be\label{e:identitiesintro}
\begin{split}
\int_{\bbR}\exp\Big(\frac{yv-\phi(y)}{\ch}\Big) \frac{1}{\sqrt{2\pi\sigma^2}}\exp\Big(-\frac{y^2}{2\sigma^2}\Big)dy&=\exp\Big(\frac{\Psi(v)}{\ch}\Big), \\
\int_{f(\bbR)}\exp\Big(\frac{yv-\Psi(v)}{\ch}\Big)\Pi(dv)&=\exp\Big(\frac{\phi(y)}{\ch}\Big), \qquad y\in \bbR,
\end{split}
 \ee
where $\Pi$ is the distribution of $V$, $v$ is an arbitrary element of the support of $\Pi$, and $\ch$ and $\sigma$ are constants. The above describes a non-linear operator whose fixed point will determine $\phi$, and, consequently $\Psi$. Under the assumption that $V$ is bounded, Theorem \ref{t:fpbd} establishes the existence of a fixed point $\Phi$ that is convex. However,  boundedness of the asset value is not a necessary condition for the existence: Theorem \ref{t:fpG} explicitly computes the fixed point when $V$ has a Gaussian distribution.

Section \ref{s:eq} studies the properties of equilibrium under the assumption that there exist a fixed point of the operator described by \eqref{e:identitiesintro}. A remarkable property of the equilibrium  is that the insider's strategy depends on the distribution of the fundamental value of the asset. This is a striking difference from the previous extensions of the Kyle model with risk neutral agents and no penalties, where the insider always employs the same bridge strategy in equilibrium irrespective of the asset distribution\footnote{A non-exhaustive list includes \cite{Back92}, \cite{CCDbp}, \cite{CCDdef}, \cite{DMB-CD},  \cite{CetRH}, and \cite{BEopttr}.}. 

The analysis of the equilibrium in Section \ref{s:eq} also reveals that all equilibrium parameters such as insider's wealth and noise traders' loss have direct representations in terms of $\Psi$, $\phi$ and its derivatives. In particular, it shows that equilibrium price is given by the solution of a heat equation with terminal value given by the first derivative of $\phi$. Moreover,  variance of the fundamental value after all trading is complete is non-zero and corresponds to an expectation involving the second derivative of $\phi$ acting on Brownian motion. 

Another remarkable consequence of the form of the solution $(U,Z)$ of \eqref{e:Uintro} and the description of the optimal strategy of the insider via $Z$ is that it reveals a connection with the  BSDE \eqref{e:Uintro} and a particular  $h$-transform of Brownian motion. Proposition \ref{p:entropy} studies this connection and shows that the expected penalty of the insider in equilibrium is a constant multiple of the relative entropy, which is finite. 

When the asset value has a Gaussian distribution, the equilibrium can be computed explicitly. The second part of Section \ref{s:eq} is devoted to the equilibrium with Gaussian payoffs. As expected, the insider's wealth is decreasing as the intensity of penalties decrease. However, this does not follow from the plausible explanation that expected penalties that are incurred in equilibrium are very high as the regulators impose higher fines. Instead, the presence of legal risk forces the insider moderate her trades and trade very little in case of high legal fines. Although the insider continue to buy (resp. sell) when the market valuation is below (resp. above) the fundamentals, the amount that she trades decrease in presence of higher fines. As a result, the expected penalties in equilibrium remain bounded even if the regulatory fines increase without bound. That the insider trades less aggressively aligns with the predictions of Shin \cite{ShinIns}, Kacperczyk and Pagnotta \cite{KPins}, as well as  
 Carr\'e et al. \cite{CCDG}), and agrees with the empirical evidence presented in \cite{KPins}.
 
 The outline of the paper is as follows. Section \ref{s:setup} presents the model assumptions. Section \ref{s:formal} gives a recipe for constructing the equilibrium and makes the connection with a quadratic BSDEs. Section \ref{s:fp} studies the fixed point of a nonlinear operator that yields the terminal condition of the BSDE via equilibrium constraints. Section \ref{s:eq} analyses the properties of equilibrium and its connection with an $h$-transform of Brownian motion. Finally, Section \ref{s:c} concludes and gives directions for future research.

\section{The setup} \label{s:setup}
As in \cite{Back92}, the trading will take place over the time interval $[0,1]$ and the risk free interest rate is set to $0$. Let $(\Omega , \cG , (\cG_t)_{t \in [0,1]} , \bbQ)$ be a filtered probability space satisfying the usual conditions.   The fundamental value of this asset equals $V$, which  will become public knowledge at time $1$. Moreover, $V$ has a finite variance with possibility of atoms.  This in particular implies the existence of a non-decreasing $f$ such that $V=f(\eta)$, where $\eta$ is a standard Normal random variable. The distribution of $V$ is denoted by $\Pi$. That is, $\bbQ(V\in dv)=\Pi(dv)$ for $v$ in the support of $V$, which is denoted by $f(\bbR)$.

Three types of agents  trade in the market. They differ in their information sets and objectives as follows and   the reader is referred to \cite{CD-GKB} for further details.

\begin{itemize}
	\item \textit{Noise/liquidity traders} are non-strategic and  trade for liquidity reasons, and
	their total demand at time $t$ is given by $\sigma B$ for a standard $(\cG_t)$-Brownian motion $
	B$ independent of $V$, and constant $\sigma>0$.
	\item \textit{Market makers} only  observe   the total demand
	\begin{equation*}
		Y=\theta+\sigma B, 
	\end{equation*}
	where $\theta$ is the demand process of the informed trader. The admissibility condition imposed later on $\theta$ will entail in particular that $Y$ is a semimartingale. 
	
	The market makers compete in a {\em Bertrand fashion} and clear the market. Similar to \cite{Back92} the market makers set the price $S$ as a
	function of the total order process at time $t$, i.e.
	\begin{equation} \label{mm:e:rule_mm}
		S_t = H(t, Y_t), \qquad \forall t\in [0,1).
	\end{equation}
	\item \textit{The informed investor} observes the price process $S$ given by
	$S_{t}=H(t, Y_t)$  as well as $V$. Different than \cite{Back92} the informed trader incurs quadratic transaction costs which in particular implies that her trading strategy is absolutely continuous, i.e. 
	\[
	d\theta_t=\alpha_t dt
	\]
	for some $\alpha$ adapted to her own filtration. The accumulated transaction cost by time $t$ is given by
	\[
	C_t:=\frac{c}{2}\int_0^t \alpha^2_s ds.
	\]
	for some $c >0$, that will be called {\em the rate} of penalties in the sequel.
	Since she is
	risk-neutral, her objective is to maximize the expected final
	wealth, i.e.
	\begin{align}
		\sup_{\theta \in \mathcal{A}}E^{0,v}\left[ W_{1}^{\theta
		}\right], \mbox{ where} \label{eq:ins_obj}
		\\
		W_{1}^{\theta
		}=
		(V-S_{1-})\theta _{1-}+\int_{0}^{1-}\theta _{s}dS_{s} - \frac{c}{2}\int_0^1 \alpha^2_s ds= \int_0^1(V-S_s)\alpha_sds- \frac{c}{2}\int_0^1 \alpha^2_s ds. \label{mm:eq:insW}	\end{align}
	In above $\mathcal{A}$ is the set of admissible trading strategies
	for the given pricing rule, which will be defined in Definition \ref{mm:d:iadm}. Moreover, $E^{0,v}$ is the expectation with respect to $P^{0,v}$, which is the regular conditional distribution of  $(B_s, V; s\leq 1)$ given $B_0=0$ and $V=v$, which exists due to Theorem 44.3 in \cite{Bauer}.
	
	The expected wealth formulation above can also arise in case the informed trader is an illegal inside trader  who faces the risk of investigation that will result in penalties. Indeed, suppose that an investigation can successfully identify illegal inside trading with probability $p$, after which the insider not only loses her gains from trade but also pay a legal penalty of $k\int_0^1\alpha^2_t dt$. The expected profit of the insider\footnote{It is assumed that the invesitigation takes place after traded are over as in the case of  a whistleblower.} under this scenario is
	\[
	E^{0,v}\bigg[(1-p)\int_0^1(V-S_s)\alpha_sds- pk\int_0^1 \alpha^2_s ds\bigg]= (1-p)E^{0,v}\bigg[\int_0^1(V-S_s)\alpha_sds- \frac{pk}{1-p}\int_0^1 \alpha^2_s ds\bigg].
	\]
	Thus the coefficient $c$ in \eqref{mm:eq:insW} can be associated with $\frac{pk}{1-p}$, which gets large as the probability of a successful investigation gets bigger. 
\end{itemize}
 Given the above market structure, one can now  define the filtrations of the market makers and of the informed trader.    First  define $\cF:=\sigma(B_t,V; t \leq 1)$.   Due to the measurability of regular conditional distributions one can define  the probability measure $\bbP$ on $(\Om, \cF)$  by
 \be \label{mm:d:bbP}
 \bbP(E)=\int_{f(\bbR)} P^{0,v}(E) \Pi(dv),
 \ee
 for any $E \in \cF$.  
 
 The market makers' filtration, denoted by $\cF^M$, will be the right-continuous augmentation with the $\bbP$-null sets of the filtration generated by $Y$. In particular $\cF^M$ satisfies the usual conditions. 
 
 On the other hand, augmented the insider's filtration with the $P^{0,v}$-null sets will create an extra dependence on the value of $V$ purely for technical reasons. To avoid this,  we assume that the informed trader's filtration, denoted by $\cF^I$, is the right continuous augmentation\footnote{See  Section 3 of \cite{GTMP} for a recipe of the procedure.} of the filtration generated by $S$ and $V$  with the sets of
 \be \label{mm:e:nullI}
 \cN^I:=\{E\subset \cF: P^{0,v}(E)=0, \, \forall z\in \bbR\}.
 \ee
 Similarly,  $\cF^{B,V}$ will denote the right continuous augmentation of the filtration generated by $B$ and $V$  with the sets of $\cN^I$. Note that the resulting filtrations are {\em not} complete. Moreover, $\cF^I=\cF^{B,V}$ when $H$ is strictly increasing in total order, which will be the case when $H$ is an {\em admissible pricing rule} per Definition \ref{mm:d:prule}. 
 
 An equilibrium is a pair consisting of
 an \emph{admissible} pricing rule and an \emph{admissible}
 trading strategy such that: \textit{a)} given the pricing rule
 the trading strategy is optimal, \textit{b)} given the trading
 strategy,  the pricing rule is {\em rational} in the following sense:
 \be \label{mm:d:mm_obj}
 H(t,Y_t)=S_t=\mathbb{E}[ V|\mathcal{F}_t^M],
 \ee
 where $\bbE$ corresponds to the expectation operator under $\bbP$.
 To formalize
 this definition of equilibrium, one needs the notions of admissible
 pricing rules and trading strategies.

 \begin{definition}\label{mm:d:prule} An {\em admissible
 		pricing rule} is any function $H$ such that  $H \in C^{1,2}([0,1) \times \bbR)\cap C([0,1] \times \bbR) $, and  $x \mapsto H (t,x)$ is strictly increasing for every $t\in [0,1]$.
 \end{definition}
 \begin{definition} \label{mm:d:iadm}
 	An $\cF^{B,Z}$-adapted $\theta$ is said to be an  admissible trading
 	strategy for a  given admissible pricing rule $H$  if the following conditions are stisfied:
 	\begin{enumerate} \item $\theta$ is  absolutely continuous and of finite variation; that is
 		\[
 		\theta_t=\int_0^t \alpha_sds \mbox{ and } \int_0^1 |\alpha_t|dt <\infty, \, P^{0,v}\mbox{-a.s.,}
 		\]
 		 for some adapted $\alpha$ for each $v \in f(\bbR)$.
 		\item No doubling strategies are
 		allowed, i.e. for all $v \in f(\bbR)$  		\begin{equation}
 		E^{0,v}\int_{0}^{1}(H^2(s,Y_s)+\alpha^2_s)ds<\infty.
 			\label{mm:e:theta_cond_2}
 		\end{equation}
 	 The set of admissible trading strategies for the  given pricing rule $H$ is denoted with $\mathcal{A}(H)$.
 	\end{enumerate}
 \end{definition} 
 \begin{remark} The square integrability condition on $H$ is standard in the literature on the Kyle-Back models. The new condition on $\alpha$ is almost redundant since the expected transaction costs become infinite when it is not satisfied.  However, it could be possible for the insider to manipulate the prices so that the gains from trading $\int_0^{1}(V-S_t)\alpha_t dt$ is bigger than the loss. The extra condition prevents such doubling strategies. 
 	
 On the other hand, the total wealth at the end of trading is given by
 \[
 \int_0^1 (V-H(t,Y_t))\alpha_tdt -\frac{c}{2} \int_0^1 \alpha^2_tdt\leq \int_0^1 (V-H(t,Y_t))^2dt +\frac{1-c}{2}\int_0^1\alpha_t^2dt.
 \]
 Thus, if the penalties are high enough, i.e. $c>1$, the insider will never use strategies that are not square integrable if the prices are. In such case, the condition on $\alpha$ is indeed redundant.
 \end{remark}
The market equilibrium can now be defined as follows.
 \begin{definition} \label{eqd} A couple $(H^{\ast}, \theta^{\ast})$
 	is said to form an equilibrium if $H^{\ast}$ is an admissible pricing rule,
 	$\theta^{\ast} \in \cA(H^{\ast})$, and the following conditions are
 	satisfied:
 	\begin{enumerate}
 		\item {\em Market efficiency condition:} given $\theta^{\ast}$,
 		$H^{\ast}$ is a rational pricing rule, i.e. it satisfies \eqref{mm:d:mm_obj}.
 		\item {\em Insider
 			optimality condition:} given $H^{\ast}$, $\theta^{\ast}$ solves
 		the insider optimization problem for all $v\in f(\bbR)$:
 		\begin{equation*}
 			E^{0,v}[W^{\theta^{\ast}}_1] = \sup_{\theta \in \cA(H^{\ast})} E^{0,v} [W^{\theta}_1]<\infty.
 		\end{equation*}
 	\end{enumerate}
 \end{definition}
 \section{A recipe for equilibrium}\label{s:formal}
 Since the market maker's pricing rule leads to the martingale property of $S$ and one typically expects the insider's strategy to be inconspicuous, i.e. $Y$ is a Brownian motion in its own filtration, in view of  the past literature, it is not unreasonable to restrict oneself to $H$ satisfying a heat equation:
 \be \label{e:Hpde}
 H_t + \frac{\sigma^2}{2}H_{yy}=0, \qquad t \in [0,1].
 \ee
 Keeping the above equation in mind a formal HJB equation for the value function of the insider will be obtained in the first part of this section. To this end define
 \[
 J(t,y) =\sup_{\alpha\in \cA(H)} E^{0,v}\bigg[\int_t^1 (v-H(u,Y_u))\alpha_udu-\frac{c}{2}\int_t^1\alpha^2_tdt\Big|Y_t=y\bigg].
  \]
 Direct calculations lead to
 \[
 J_t +\frac{\sigma^2}{2} J_{yy} +\sup_{\alpha} \Big\{\alpha (J_y +v-H) -\frac{c\alpha^2}{2}\Big\}=0.
 \] 
 Note that the presence of the penalty  factor $c>0$ yields that the optimal $\alpha$ can be obtained in the feedback form:
 \be \label{e:alphop}
 \alpha^* = \frac{J_y(t,y) + v -H(t,y)}{c}.
 \ee
 This leads to the following HJB equation:
 \be \label{e:HJB}
 J_t +\frac{\sigma^2}{2}  J_{yy} + \frac{ (J_y +v-H)^2}{2c}=0.
 \ee
 Recall that in the classical Kyle model $c=0$, i.e. no penalty,  the optimal strategy cannot be obtained in the feedback form from the HJB equation and instead one has $J_y= H-z$, which is one should expect when letting $c\rar 0$ in \eqref{e:alphop} to arrive at a finite limit (see Section 6.2 in \cite{DMB-CD}).
 
 Next suppose there exists a smooth function $J^0$ such that 
 \be \label{e:J_0}
 \begin{split}
 	J^0_t + \frac{\sigma^2}{2} J^0_{yy}&=0,\\
 	J^0_y &= H-v,\\
 	J^0(1,y)&=\int_{h^{-1}(v)}^y (h(x)-v)dx+ \mbox{ constant},
 \end{split}
 \ee
 for some $h$ to be determined and a constant possibly depending on $v$. Such a function exists and is the value function of the insider in the classical Kyle model with $c=0$ whenever the pricing rule (not necessarily the equilibrium one)  satisfies  \eqref{e:Hpde} (see Theorem 5.1 in \cite{CD-GKB}). In this case, $h=H(1,\cdot)$ and the  constant is zero.
 
 Thus, if one defines $u= J-J^0$ and conjectures that $J(1,\cdot)\equiv 0$, one obtains
 \be \label{e:pdeu}
 u_t + \half \sigma^2u_{yy}+ \frac{1}{2c}u_y^2=0, \quad u(1,y)=-j^0(y,v):=-J^0(1,y).
 \ee
  Note that $J^0$ would be the value function of the insider had there been no penalties on her trading strategies, i.e. $c=0$. Thus, $-u$ can be viewed as the expected additional transaction costs (or legal penalties) to the insider in equilibrium.

 It is now a common folklore that the above has a backward stochastic differential equation (BSDE) formulation:
 \be \label{e:BSDE}
 dU_t=\sigma Z_t dB_t -\frac{1}{2c} Z_t^2 dt, \quad U_1= u(1,\sigma B_1).
 \ee
 The above quadratic BSDE is known to have an explicit solution, which in turn yields
 \be \label{e:usoln}
 u(t,x)= c\sigma^2\log \rho(t,x,v), \quad \rho(t,x,v):=E^{0,v}\Big[\exp\Big(-\frac{j^0(\sigma B_1,v)}{c\sigma^2}\Big)\Big|\sigma B_t=x\Big].
 \ee
 It immediately  follows from Theorem 4.3.6 in \cite{KS} that for each $v \in f(\bbR)$,  $\rho(\cdot, \cdot, v) \in  C^{1,2}([0,1)\times \bbR])$..

 In view of  \eqref{e:alphop} and  \eqref{e:J_0}, the above implies the optimal control of the insider at time $t \in [0,1]$ is $\alpha^*(t,Y_t,V)$, where 
 \be \label{e:alpha*}
 \alpha^*(t,y,v):=\sigma^2\frac{\rho_y(t,y,v)}{\rho(t,y,v)}.
 \ee
 Recall that the above assumed that the pricing rules satisfied  \eqref{e:Hpde}. Combined with the requirement that prices follow martingales in market makers' filtration in equilibrium this in particular yields that $Y$ must be a Brownian motion in its own filtration.  Observe that if $\kappa(t) \rho(t, Y_t, v)\Pi(dv) =\bbP(V\in dv|\cF^Y_t)$ for some constant $\kappa(t)$ for each $t$, then $Y$ is a Brownian motion in its own filtration (see, e.g., Corollary 3.1 in \cite{DMB-CD}) since
 \[
 \bbE\Big[\frac{\rho_y(t,Y_t,V)}{\rho(t,Y_t,V)}\Big|\cF^Y_t
 \Big]=\int_{f(\bbR)}\rho_y(t,Y_t,v)\kappa(t)dv=\frac{d}{dy}\int_{f(\bbR)}\rho(t,y,v)\kappa(t)dv\Big|_{y=Y_t}=0,
  \]
  provided one can interchange the order of integration and differentiation.
  
 To this end, suppose that  $\rho^0(t,y,v):=\kappa(t)\rho(t,y,v)$ is the conditional density of $V$  with respect to the measure $\Pi$ given $\cF^Y_t$ and $Y_t=y$. Then, in view of Theorem 3.3 in \cite{DMB-CD}, one expect $\rho^0(\cdot,\cdot, v)$ to satisfy for each $v$
 \be \label{e:rho0}
 \rho^0_t +\half \sigma^2 \rho^0_{yy}=0.
 \ee
 On the other hand,  $\rho$ satisfies \eqref{e:rho0} by its construction. Therefore, $\frac{\kappa'}{\kappa}\equiv 0$, and consequently, $\kappa$ must be constant since it must be strictly positive. This leads to the following:
 \begin{proposition}\label{p:rhodens}
Suppose that   there exists a continuous function $j^0: \bbR\times f(\bbR)\to \bbR$ such that $\kappa \rho(0,0,\cdot)\equiv 1$ for some normalising constant $\kappa$, where $\rho$ is defined via \eqref{e:usoln}. Assume further that there exists a unique strong solution on $(\Om, \cG, (\cG_t)_{t\in [0,1]}, \bbQ)$ to 
\[
Y_t= B_t +\int_0^t \frac{\rho_y(s,Y_s,V)}{\rho(s,Y_s,V)}ds
\]
such that 
\be \label{e:Zsquareint}
\bbE^{\bbQ}\Big[\int_0^t\Big(\frac{\rho_y(s,Y_s,V)}{\rho(s,Y_s,V)}\Big)^2\Big]<\infty, \quad t\in [0,1].
\ee
Then, 
\[
\kappa \rho(t,Y_t,v)\Pi(dv)= \bbP(V\in dv|\cF^Y_t), \qquad t\in [0,1].
\]
 \end{proposition}
\begin{proof}
	This is an immediate consequence of Theorem 3.3 in \cite{DMB-CD}. Note that the process $Z$ therein coincides with $V$ at all times under the setting of the present paper. Thus, the  martingale problem $\tilde{A}$, i.e., the martingale problem for the pair $(Y, V)$ is trivially well-posed. 
\end{proof}
 The seemingly strong integrability condition \eqref{e:Zsquareint} is in fact a condition on the square integrability of $Z$ in the BSDE representation \eqref{e:BSDE}. A minimal integrability condition as in the next result ensures that  holds.
 \begin{proposition} \label{p:Zui}
 	Consider $\rho$ defined by \eqref{e:usoln} for some continuous $j^0$. Assume that 
 	\be \label{e:Zui}
 	\int_{f(\bbR)}E^{0,v}[|j^0(\sigma B_1,v)|]\Pi(dv)<\infty.
 	\ee
 	Then, \eqref{e:Zsquareint} holds.
 \end{proposition}
\begin{proof}
To ease notation take $\sigma=1$.	As observed above, it suffices to show that 
	\[
	\bbE^{\bbQ}\Big[\int_0^1Z_t^2 dt\Big]<\infty,
	\]
	where $(U,Z)$ is as in \eqref{e:BSDE}. It follows from the Jensen's inequality and \eqref{e:usoln} that 
	\be \label{e:boundU}
	U_t \geq  M_t(v):= -E^{0,v}\big[j^0(B_1,v)\big|B_t\big] , \; P^{0,v}\mbox{-a.s..}.
	\ee
	Since $M(v)$ is a martingale on the time interval $[0,1]$, we have that $U$ is  bounded from below by a  uniformly integrable  process. Thus, by Fatou's lemma $\int_0^{\cdot}Z_sdB_s$ in \eqref{e:BSDE} is a supermartingale.  Therefore,
	\[
	\frac{1}{2c}E^{0,v}\Big[\int_0^1Z_t^2 dt\Big]\leq E^{0,v}[U_0- U_1]\leq U_0+  E^{0,v}[j^0(B_1,v)]<\infty.
	\]
	This yields the claim by the independence of $B$ and $V$ under $\bbQ$, and the hypothesis \eqref{e:Zui}.
\end{proof}
The above considerations give the recipe for constructing an equilibrium:
\begin{enumerate}
	\item Find  a continuous function $j^0: \bbR\times f(\bbR)\to \bbR$ such that i)  $\rho(1,y,v)\Pi(dv)$ is a probability measure for each $y$, where $\rho$ is given by \eqref{e:usoln}, ii) $\kappa \rho(0,0,\cdot)\equiv 1$ for some constant $\kappa>0$, and iii) it is differentiable in its first parameter with  $j^0_y(y,v)= h(y)-v$. Note that the second condition entails
	\be \label{e:initialden}
	\kappa \int_{\bbR} \frac{1}{\sqrt{2\pi\sigma^2}}\exp\Big(-\frac{y^2}{2\sigma^2}\Big)\exp\Big(-\frac{j^0(y,v)}{\sigma^2 c}\Big)dy=1, \qquad \forall v\in f(\bbR).
	\ee
	\item Set 
	\[
	H_t+\half H_{yy}=0, \qquad H(1,y)=h(y).
	\]
	\item Show that $(H,\theta^*)$  with $d\theta^*_t= \frac{\rho_y(t,Y_t,V)}{\rho(t,Y_t,V)}dt$ is equilibrium provided they are admissible by using the candidate value function $J=J^0-u$  that  satisfies \eqref{e:HJB}.
\end{enumerate}
\begin{remark}In view of the rationality of the  pricing rule, i.e. \eqref{mm:d:prule}, in equilibrium  the above recipe appears to assume that  the function $h$ also satisfies 
\[
h(y)= \kappa\int_{f(\bbR)}v\rho(1,y,v)\Pi(dv)= \kappa\int_{f(\bbR)}v\exp\Big(-\frac{j^0(y,v)}{c\sigma^2}\Big)\Pi(dv).
\]
However, this already follows from the properties of $j^0$ in Step (1) alone. Indeed,  since $\kappa \rho(1,y,\cdot)$ is a proper density for each $y$, 
\[
0=\frac{d}{dy}1= \frac{d}{dy}\kappa \int_{f(\bbR)}\exp\Big(-\frac{j^0(y,v)}{\sigma^2c}\Big)\Pi(dv)=\frac{\kappa}{\sigma^2c} \int_{f(\bbR)}(v-h(y))\exp\Big(-\frac{j^0(y,v)}{\sigma^2c}\Big)\Pi(dv),
\]
provided one can  differentiate under the integral sign.  Thus,
\[
h(y)= h(y)\kappa \int_{f(\bbR)}\exp\Big(-\frac{j^0(y,v)}{c\sigma^2}\Big)\Pi(dv)= \kappa \int_{f(\bbR)}v\exp\Big(-\frac{j^0(y,v)}{c\sigma^2}\Big)\Pi(dv)
\]
as required.
\end{remark}
\section{A fixed point algorithm}\label{s:fp}
The recipe  from the previous section shows that the equilibrium mainly boils down to finding a function $j^0$ with certain properties. The condition on its derivative implies 
\be \label{e:j0decomp}
j^0(y,v)= \Psi(v) +\phi(y)-yv,
\ee
where $\psi$ and $\phi$ are continuous functions on $f(\bbR)$ and $\bbR$, respectively. This section will describe an operator on the space of continuous functions, whose fixed point will determine the functions $\psi$ and $\phi$ that will appear in equilibrium. Although the main theorem of this section will be proved under a stronger assumption, the condition below is sufficient for a number of important estimates. 
\begin{assumption} \label{a:mgf}
All exponential moments of $V$ exist. That is, 
\be \label{e:expmV}
M_V(r):=\int_{f(\bbR)}e^{rv}\Pi(dv)<\infty, \qquad \forall r\in \bbR.
\ee
\end{assumption}
The following notation will also be used to ease the exposition:

\noindent {\bf Notation: } $\ch:=c\sigma^2$.
 
Since $\exp(-\frac{1}{\ch}j^0(y,\cdot))\kappa$ will be a probability density function itself for all $y$ for some constant $\kappa$ independent of $y$, one can without loss of generality take $\kappa=1$ by incorporating it into $\Psi$ or $\phi$. This yields the first relationship:
\[
\int_{f(\bbR)}\exp\Big(\frac{yv-\Psi(v)}{\ch}\Big)\Pi(dv)=\exp\Big(\frac{\phi(y)}{\ch}\Big), \qquad y\in \bbR.
\]
The second relationship follows from the initial condition \eqref{e:initialden}:
\[
\int_{\bbR}\exp\Big(\frac{yv-\phi(y)}{\ch}\Big) \frac{1}{\sqrt{2\pi\sigma^2}}\exp\Big(-\frac{y^2}{2\sigma^2}\Big)dy=\exp\Big(\frac{\Psi(v)}{\ch}\Big), \qquad v\in f(\bbR).
\]
Thus one obtains the following two operators on the space of measurable functions:
\be\label{e:identities}
\begin{split}
T_1 (\phi) (v):=\int_{\bbR}\exp\Big(\frac{yv-\phi(y)}{\ch}\Big) \frac{1}{\sqrt{2\pi\sigma^2}}\exp\Big(-\frac{y^2}{2\sigma^2}\Big)dy&=\exp\Big(\frac{\Psi(v)}{\ch}\Big), \qquad v\in f(\bbR);\\
T_2 (\Psi)(y):=\int_{f(\bbR)}\exp\Big(\frac{yv-\Psi(v)}{\ch}\Big)\Pi(dv)&=\exp\Big(\frac{\phi(y)}{\ch}\Big), \qquad y\in \bbR
\end{split}
\ee
Consequently, one can define an operator $T_0$ on the space of measurable functions by 
\be \label{e:defT0}
T_0\phi =\ch \log T_2(T_1(\phi)),
\ee
whose domain consists of functions so that $T_0(\phi)$ is finite for all $y\in \bbR$. Once this is done, one can apply a normalisation so that $\phi(0)=0$. This leads to the operator $T$ defined on the domain of $T_0$ as follows:
\be \label{e:defT}
T\phi =\ch \log \frac{T_2(T_1(\phi))}{T_2(T_1(\phi))(0)}.
\ee
By similar reasoning the domain $\cD(T_i)$ of the operator $T_i$ consist of functions $g$ such $T_ig$ takes values in $(0,\infty)$.  Consequently, 
\[
\cD(T)=\cD(T_0)=\{\phi\in  \cD(T_1): T_1(\phi)\in   \cD(T_2)\}.
\]
Due to the extra normalisation involved it is not immediately obvious that the fixed points of $T$ and $T_0$ are related. However, as the next proposition shows, there is a one-to-one correspondence. 
\begin{proposition}
	\label{p:TT0}
	There exists a fixed point for $T_0$ if and only if there exists a fixed point for $T$. In particular, any fixed point of $T$ is a fixed point of $T_0$.
\end{proposition}
\begin{proof}
By direct calculations,  if $T_0\phi=\phi$, $T\phi_0=\phi_0$ for $\phi_0=\phi-\phi(0)$. 

Now, suppose $T\phi=\phi$. Then, $T_2 \Psi= k \exp(\frac{\phi}{\ch})$ for some $k$, where $\Psi=\ch \log T_1\phi$. That is,
\[
\begin{split}
\int_{\bbR}\exp\Big(\frac{yv-\phi(y)-\Psi(v)}{\ch}\Big) \frac{1}{\sqrt{2\pi\sigma^2}}\exp\Big(-\frac{y^2}{2\sigma^2}\Big)dy&=1, \qquad v\in f(\bbR);\\
\int_{f(\bbR)}\exp\Big(\frac{yv-\Psi(v)-\phi(y)}{\ch}\Big)\Pi(dv)&=k, \qquad y\in \bbR.
\end{split}
\]
Integrating both sides of the first identity with respect to $\Pi$ yields
\[
k\int_{\bbR} \frac{1}{\sqrt{2\pi\sigma^2}}\exp\Big(-\frac{y^2}{2\sigma^2}\Big)dy=1,
\]
in view of the second identity in the previous display and Fubini's theorem. Thus, $k=1$ and $T_0\phi=\phi$.
\end{proof}
The next result shows that the fixed point of $T$, if it exists, should be a convex function.
\begin{proposition} \label{p:psilb}
Consider $\phi \in \cD(T_1)$ and let $\Psi=\ch\log T_1 \phi$. Then, the following statements are valid:
\begin{enumerate}
	\item If  $\bbE[|\phi(\sigma B_1)|]<\infty$,  $\Psi$ is bounded from below:
	\[
	\Psi\geq -\int_{\bbR}\phi(y) \frac{1}{\sqrt{2\pi\sigma^2}}\exp\Big(-\frac{y^2}{2\sigma^2}\Big)dy.
	\]
	\item Suppose that $\Psi$ is bounded from below and Assumption \ref{a:mgf} holds. Then $\phi \in \cD(T)$.

	Moreover, $T\phi$ is infinitely  differentiable and strictly convex. In particular,
	\begin{align} 
	\frac{d}{dy}T\phi(y)&=\frac{\int_{f(\bbR)}v\exp\Big(\frac{yv-\Psi(v)}{\ch}\Big)\Pi(dv)}{\int_{f(\bbR)}\exp\Big(\frac{yv-\Psi(v)}{\ch}\Big)\Pi(dv)}, \label{e:Tphidrv}\\
	\frac{d^2}{dy^2}T\phi(y)&=\frac{1}{\ch}\frac{\int_{f(\bbR)}v^2\exp\Big(\frac{yv-\Psi(v)}{\ch}\Big)\Pi(dv)}{\int_{f(\bbR)}\exp\Big(\frac{yv-\Psi(v)}{\ch}\Big)\Pi(dv)}-\frac{1}{\ch}\Bigg(\frac{\int_{f(\bbR)}v\exp\Big(\frac{yv-\Psi(v)}{\ch}\Big)\Pi(dv)}{\int_{f(\bbR)}\exp\Big(\frac{yv-\Psi(v)}{\ch}\Big)\Pi(dv)}\Bigg)^2. \label{e:Tphidrv2}
	\end{align}
	 \end{enumerate}
\end{proposition}
\begin{proof}
\begin{enumerate}
	\item This is a straightforward consequence of Jensen's inequality applied to the natural logarithm and definition of $T_1$ in \eqref{e:identities}.
	\item Since $\phi\in \cD(T_1)$, $\Psi$ is finite. Thus, $T_2\Psi$ never vanishes. Moreover the lower bound on $\Psi$ yields for some constant $K$ that
	\[
	T_2\Psi(y)\leq K \int_{f(\bbR)}\exp\Big(\frac{yv}{\ch}\Big)\Pi(dv)<\infty.
		\]
 Thus, $T_0\phi$ is finite everywhere. This yields the claim.

Note that since the moment generating function are finite everywhere, it is smooth and in particular (see, e.g.,  (21.24) in Section 21 and the preceding discussion in \cite{Bil} )
\[
M_V^{(n)}=\int_{f(\bbR)}v^ne^{rv}\Pi(dv)<\infty, \qquad \forall r\in \bbR \mbox{ and } n
\geq 0.
\] 
Thus,  the above  integrability, the lower bound on $\Psi$, and  the dominated convergence theorem immediately yield \eqref{e:Tphidrv}. By iterating this operation one obtains that the numerator in \eqref{e:Tphidrv} is infinitely differentiable by the smoothness of $M_V$.   

Similarly, the second derivative is given by \eqref{e:Tphidrv2}, which is  positive being a positive multiple of the variance of a random variable $\tilde{v}$ with the probability density
\[
P(\tilde{v}\in dx)=\frac{\exp\Big(\frac{yx-\Psi(x)}{\ch}\Big)\Pi(dx)}{\int_{f(\bbR)}\exp\Big(\frac{yv-\Psi(v)}{\ch}\Big)\Pi(dv)}, \qquad x\in f(\bbR).
\]
Since $\tilde{v}$ is not constant, the second derivative is strictly positive, hence the strict convexity follows.
\end{enumerate}
\end{proof}
An analogous result, whose proof is omitted, also holds for $\Psi$. To present its statement let's introduce
\be \label{e:defTT}
\tilde{T}\Psi =\ch \log \frac{T_1(T_2(\Psi))}{T_1(T_2(\Psi))(0)}.
\ee
\begin{proposition}
	Consider $\Psi \in \cD(T_2)$ and let $\tilde{\phi}=\ch\log T_2 \Psi$. Then, the following statements are valid:
	\begin{enumerate}
		\item If  $\bbE[|\Psi(V)|]<\infty$,  $\tilde{\phi}$ is bounded from below:
		\[
	\tilde{\phi}(y)\geq y\bbE[V]-\bbE[\Psi(V)].
		\]
		\item Suppose that $\tilde{\phi}$ is bounded from below. Then $\Psi \in \cD(\tilde{T})$. Moreover, $\tilde{T}\Psi$ is infinitely  differentiable and strictly convex. In particular,
		\begin{align} 
			\frac{d}{dv}\tilde{T}\Psi(v)&=\frac{\int_{\bbR}y\exp\Big(\frac{yv-\tilde{\phi}(y)}{\ch}\Big)p(\sigma,y)dy}{\int_{\bbR}\exp\Big(\frac{yv-\tilde{\phi}(y)}{\ch}\Big)p(\sigma,y)dy}, \label{e:Tpsidrv}\\
			\frac{d^2}{dv^2}\tilde{T}\Psi(v)&=\frac{\int_{\bbR}y^2\exp\Big(\frac{yv-\tilde{\phi}(y)}{\ch}\Big)p(\sigma,y)dy}{\int_{\bbR}\exp\Big(\frac{yv-\tilde{\phi}(y)}{\ch}\Big)p(\sigma,y)dy}-\Bigg(\frac{\int_{\bbR}y\exp\Big(\frac{yv-\tilde{\phi}(y)}{\ch}\Big)p(\sigma,y)dy}{\int_{\bbR}\exp\Big(\frac{yv-\tilde{\phi}(y)}{\ch}\Big)p(\sigma,y)dy}\Bigg)^2, \label{e:Tpsidrv2}
		\end{align}
		where $p(\sigma,\cdot)$ is the density of a normal distribution with variance $\sigma^2$.
	\end{enumerate}
\end{proposition}
The above results indicate that the fixed point of $T$ is likely  to be a smooth convex function. The next lemma improves upon the previous proposition by providing an upper bound on the function $\Psi$. 
\begin{lemma} \label{l:psiub}
	Consider a convex and differentiable $\phi \in \cD(T_1)$ with $\phi(0)=0$,  and let $\Psi=\ch\log T_1 \phi$.  Then
	\be\label{e:psiUB}
	\Psi(v)\leq \frac{\big(v-\phi'(0)\big)^2}{2 c^2\sigma^2}, \qquad v\in f(\bbR).
	\ee
\end{lemma}
\begin{proof}
	Observe that $\phi(y) \geq y\phi'(0)$ for all $y\in \bbR$. Thus,
	\[
	T_1\phi(v) \leq \int_{\bbR} \exp\Big(\frac{y(v-\phi'(0))}{\ch}\Big) \frac{1}{\sqrt{2\pi\sigma^2}}\exp\Big(-\frac{y^2}{2\sigma^2}\Big)dy=\exp\Big(\frac{\big(v-\phi'(0)\big)^2\sigma^2}{2\ch^2}\Big),
	\]
	in view of the moment generating function of the normal distribution. This proves the upper bound.
	
%
\end{proof}

These estimates leads to the following theorem that establishes the existence of a fixed point when $V$ is bounded.
\begin{theorem} \label{t:fpbd}
   Suppose that $\kappa:=\sup\{|v|:v\in f(\bbR)\}<\infty$. Then, there exists a strictly convex $\phi$  such that $T\phi =\phi=T_0\phi$. 
\end{theorem}
\begin{proof}
Let $\phi_0=0$ and define $\phi_{n+1}=T\phi_n$ until  $\phi_n$ no longer belongs to $\cD(T)$.
\begin{enumerate}
	\item[Step 1] {\em $\phi_n \in \cD(T)$, $\phi_n(0)=0$, and $|\phi'_n|\leq \kappa$ for all $n\geq 0$}: Denoting $T_1 \phi_0$ by $\Psi_0$, Proposition \ref{p:psilb} first shows that $\Psi_0\geq -\bbE[\phi_0(\sigma B_1)]=0$. Since Assumption \ref{a:mgf} is clearly satisfied, this in turn implies $\phi_0\in \cD(T)$.  Clearly, $|\phi'_0|\leq \kappa$. 
	
	Next, suppose $\phi_n \in \cD(T)$, $\phi_n(0)=0$ and $|\phi'_n|\leq \kappa$ for some $n$. Then, $\phi_{n+1}(0)=0$ by the construction of $T$ and  $|\phi_{n+1}'|\leq \kappa$ by \eqref{e:Tphidrv}. Thus, $|\phi_{n+1}(y)|\leq \kappa|y|$, implying $\phi_{n+1} \in \cD(T_1)$ and $E[|\phi_{n+1}(\sigma B_1)|]<\infty$. This yields that  $\phi_{n+1} \in \cD(T)$ by another application of Proposition \ref{p:psilb}. Therefore, the claim follows from induction.
	\item[Step 2] {\em Existence of a convergent subsequence:} Since the collection $(\phi_n)_{n\geq 1}$ is equicontinuous and  bounded uniformly by the function $y \mapsto \kappa |y|$, one can extract a subsequence that converges (uniformly in compacts) to some $\phi$.
	\item[Step 3:] {\em Fixed point property:} With an abuse of notation suppose $(\phi_n)$ converges and let $\Psi_n:= T_1 \phi_n$. Note that the dominated convergence theorem implies that $\Psi:=\lim_{n\rar \infty}T_1 \phi_n=T_1 \phi$. Next, Proposition \ref{p:psilb} shows that
	\[
	\Psi_n \geq -\bbE[\phi_n(\sigma B_1)]\geq -\kappa \sigma  \bbE[|B_1|].
	\]
	Moreover, Lemma \ref{l:psiub} yields that $\Psi_n \leq \frac{\kappa^2}{c^2\sigma^2}$. Thus, $(\Psi_n)$ is uniformly bounded. Another application of the dominated convergence theorem now yields
	\[
	\lim_{n\rar \infty}T_2\Psi_n/T_2\Psi_n(0)= T_2\Psi/T_2\Psi(0)=\exp\Big(\frac{T\phi}{\ch}\Big)
	\]
On the other hand, 
\[
T_2\Psi_n/T_2\Psi_n(0)=\exp\Big(\frac{\phi_{n+1}}{\ch}\Big).
\]
Combining the last two displays and taking limits establish that $\phi= T\phi$. The strict convexity follows from Proposition \ref{p:psilb}. The remaining assertion is a consequence of Proposition \ref{p:TT0}.
\end{enumerate} 
\end{proof}
\begin{example}
	Consider the case of a Bernoulli $V$, where $\Pi(\{1\})=1-\Pi(\{0\})=p$. Existence of a fixed point for $T$  is  ensured by Theorem \ref{t:fpbd}.  As the support contains only two points, an easier representation of $\phi$ can be found as follows.
	
	In view of the normalisation that $\Phi(0)=0$, one can conjecture that
	\[
	\exp\Big(\frac{\phi}{\ch}\Big)=\exp\Big(\frac{y}{\ch}\Big)ap +1-ap,\mbox{ and } \Psi(1)=-\ch \log a, \Psi(0)=-\ch\log \frac{1-ap}{1-p}.
	\]
	Thus, it remains to pinpoint the value of $a$, which is  fixed by the following non-linear equation.
	\[
		\frac{1}{a}=\int_{\bbR}\frac{e^{\frac{y}{\ch}}}{ap(e^{\frac{y}{\ch}}-1)+1}\frac{1}{\sqrt{2\pi \sigma^2}}\exp\Big(-\frac{y^2}{2 \sigma^2}\Big)dy.
	\]
	However, this is equivalent to 
	\[
1=\int_{\bbR}\frac{e^{\frac{y}{\ch}}}{p(e^{\frac{y}{\ch}}-1)+a^{-1}}\frac{1}{\sqrt{2\pi \sigma^2}}\exp\Big(-\frac{y^2}{2 \sigma^2}\Big)dy.
	\]
	Note that $a \in (0,p^{-1})$.  Thus, the right hand side is increasing from $0$ to $1/p$ on $(0, p^{-1})$, and  there exists a unique solution belonging to $(0,p^{-1})$.
\end{example}
\section{Equilibrium}\label{s:eq}
In the first part of this section, existence of equilibrium for the economy described in Section \ref{s:setup}  will be established under  the following assumption.
\begin{assumption}
	\label{a:fpexists}  There exists a pair of continuous functions $(\phi^*,\Psi^*)$ such that $\phi^*$ is a fixed point of the operator $T$ in \eqref{e:defT} and $\Psi^*=c\sigma^2 \log T_1\phi^*$, where $T_1$ is as in \eqref{e:identities}. Moreover, $\phi^*$ is twice continuously differentiable satisfying \eqref{e:Tphidrv} and \eqref{e:Tphidrv2}, and \eqref{e:Zui} is satisfied when $j^0$ is replaced by 
	\be \label{e:j*}
	j^*(y,v):=\Psi^*(v) + \phi^*(y)-yv.
	\ee
\end{assumption}
Recall from Proposition \ref{p:psilb} that $\phi^*$ is strictly convex, which will be used frequently in proofs.  

In view of the recipe at the end of Section \ref{s:formal} the candidate pricing rule for equilibrium is the following:
\be \label{e:defH*}
H^*(t,y)=\int_{f(\bbR)}\rho^*(t,y,z)z\Pi(dz),
\ee
where
\be\label{e:rho*}
\rho^*(t,y,v):=E^{0,v}\Big[\exp\Big(-\frac{j^*(\sigma B_1,v)}{\ch}\Big)\Big|\sigma B_t=y\Big].
\ee 
 Observe that the expectation above is non-zero and    finite due to the following:
\begin{lemma}\label{l:rho} The following bounds apply:
	\begin{align*}
		E^{0,v}\Big[\exp\Big(-\frac{j^*(\sigma B_1,v)}{\ch}\Big)\Big|\sigma B_t=x\Big]&\leq \exp\Big(-\frac{\Psi^*(v)}{\ch}+ \frac{(1-t)(v-h^*(0))^2}{2c^2\sigma^2}+(v-h^*(0)x\Big),\\
		E^{0,v}\Big[\exp\Big(-\frac{j^*(\sigma B_1,v)}{\ch}\Big)\Big|\sigma B_t=x\Big]&\geq \exp\Big(\frac{vx -\Psi^*(v)-E^{0,v}[\phi^*(\sigma B_1)|\sigma B_t=x]}{\ch} \Big).
	\end{align*}
\end{lemma}	
\begin{proof}
	Since $\phi^*$ is convex by Proposition \ref{p:psilb} and vanishes at $0$,  $\phi^*(y)\geq y h^*(0)$. Thus,
	\begin{align*}
		E^{0,v}\Big[\exp\Big(-\frac{j^*(\sigma B_1,v)}{\ch}\Big)\Big|\sigma B_t=x\Big]&=\exp\Big(-\frac{\Psi^*(v)}{\ch}\Big)	E^{0,v}\Big[\exp\Big(\frac{v\sigma B_1-\phi^*(\sigma B_1)}{\ch}\Big)\Big|\sigma B_t=x\Big]\\
		&\leq \exp\Big(-\frac{\Psi^*(v)}{\ch}\Big)	E^{0,v}\Big[\exp\Big(\frac{\sigma B_1(v-h^*(0))}{\ch}\Big)\Big|\sigma B_t=x\Big],
	\end{align*}
	which yields the first claim in view of the moment generation function of normal distribution. The second claim is a direct consequence of Jensen's inequality.
\end{proof}
Thus, \eqref{e:Tphidrv} implies in particular that
\be \label{e:pdeH*}
H^*_t +\frac{\sigma^2}{2}H^*_{yy}=0, \qquad H^*(1,\cdot)=h^*=\frac{d\phi^*}{dy}.
\ee
Moreover, the discussion in Section \ref{s:formal} predicts  that the value function of the informed trader is given by $J=J^0+u$, 
\be \label{e:J0}
J^0_t + \frac{\sigma^2}{2}J^0_{yy}=0, \quad J^0(1,y)=j^*(y,v),
\ee
and $u$ is given by \eqref{e:usoln} with $j^0$ replaced by $j^*$.
These considerations yield the following  in view of the PDE \eqref{e:pdeu} that $u$ satisfies.
\begin{proposition}
	Under Assumption \ref{a:fpexists} there exists a classical solution to 
	\be \label{e:J}
	J_t + \frac{\sigma^2}{2}J_{yy} + \frac{(J_y + v-H^*)^2}{2c}=0, \qquad J(1,\cdot)=0,
	\ee
	where $H^*$ is given by \eqref{e:defH*}. In particular, $J=J^0+ \ch \log \rho^*$, where $J^0$ is given by \eqref{e:J0}\footnote{When there are no penalties, the decomposition of $J$ coincides with the one given in Theorem 3.2 in \cite{BEopttr}.}.
\end{proposition}
As expected from the formal calculations of Section \ref{s:formal}, the function $J$ above can be used to show that an equilibrium is given by the pricing rule $H^*$ and the strategy $\theta^*$ given by 
\be \label{e:theta*}
d\theta^*_t=\sigma^2 \frac{\rho^*_y(t,Y_t,V)}{\rho^*(t,Y_t,V)}dt, t\in [0,1], \mbox{ and } \theta^*_0=0.
\ee
The admissibility condition that is required from any equilibrium candidate is square integrability of the strategy as in \eqref{mm:e:theta_cond_2}. In case of the above candidate this is intimately related to the entropy af an $h$-transform.
\begin{proposition}\label{p:entropy}
	Suppose that Assumption \ref{a:fpexists} holds.  For each $v$, there exists a unique strong solution to 
	\be \label{e:Yeq}
	Y^*_t =\sigma B_t +\sigma^2  \int_0^t\frac{\rho^*_y(s,Y^*_s,v)}{\rho^*(s,Y_s,v)}ds,
	\ee
where $\rho^*$ is given by \eqref{e:rho*}. 

Let $\bbQ^*$ be the law induced by $Y^*$ on the space of continuous functions on $[0,1]$ vanishing at $0$; that is $C_0([0,1])$. Let $(\cB_t)_{t\in [0,1]}$ be the right continuous augmentation of the natural filtration of the coordinate process $X$, and $\bbW$ be the Wiener measure. Then $\bbQ^*$ is given by the following $h$-transform of Brownian motion:
\be \label{e:htransform}
\bbE^{\bbQ^*}[F]=\frac{\bbE^{\bbW} \Big[F \exp\Big(\frac{v\sigma X_1-\phi^*(\sigma X_1)}{\ch}\Big)\Big]}{\bbE^{\bbW} \Big[\exp\Big(\frac{v\sigma X_1-\phi^*(\sigma X_1)}{\ch}\Big)\Big]}, \quad F\in\cB_1.
\ee
Moreover, 
\begin{align} \label{e:entropy}
\frac{1}{2}E^{0,v}\bigg[\sigma^2  \int_0^1\Big(\frac{\rho^*_y(s,Y^*_s,v)}{\rho^*(s,Y_s,v)}ds\Big)^2\bigg]&=H(\bbQ^*||\bbW):= \bbE^{\bbW}\Big[\frac{d\bbQ^*}{d\bbW}\log \frac{d\bbQ^*}{d\bbW}\Big]\\
&=\frac{v(\Psi^*)'(v)-\Psi^*(v)-E^{0,v}[\phi^*(Y^*_1)]}{\ch}<\infty.\nn
\end{align}
\end{proposition}	
\begin{proof}
Existence of  a unique strong solution to \eqref{e:Yeq} follows from  Theorem 2.2.7 in \cite{DMB-CD} since for each $v$, $\rho^*(\cdot, \cdot,v) \in C^{1,2}([0,1],\bbR)$ and $\rho^*$ never vanishes. Note that the claimed smoothness is valid even at time $t=1$ since $\phi^*$ is smooth by Assumption \eqref{a:fpexists}.  Then, \eqref{e:htransform} follows from Girsanov's theorem after noticing that  $\rho(1, y,v)=\exp(-\frac{\Psi(v)}{\ch}) \exp(\frac{yv -\phi(y)}{\ch})$ so that $\exp(-\frac{\Psi(v)}{\ch})$ drops out in change of measure calculations.

Thus, the equality in \eqref{e:entropy} is a consequence of \eqref{e:htransform} and the definition of relative entropy. To compute its desired expression, note that
\[
\frac{d\bbQ^*}{d\bbW}= \exp\Big(\frac{v\sigma X_1-\phi^*(\sigma X_1)-\Psi^*(v)}{\ch}\Big).
\]
. Thus,
\[
\begin{split}
	H(\bbQ^*||\bbW)&= \bbE^{\bbW}\Big[\frac{v\sigma X_1-\phi^*(\sigma X_1)-\Psi^*(v)}{\ch}\exp\Big(\frac{v\sigma X_1-\phi^*(\sigma X_1)-\Psi^*(v)}{\ch}\Big)\Big]\\
	&=\frac{v(\Psi^*)'(v)-\Psi^*(v)}{\ch}- \bbE^{\bbW}\Big[\frac{\phi^*(\sigma X_1)}{\ch}\exp\Big(\frac{v\sigma X_1-\phi^*(\sigma X_1)-\Psi^*(v)}{\ch}\Big)\Big]\\
	&=\frac{v(\Psi^*)'(v)-\Psi^*(v)-E^{0,v}[\phi^*(Y^*_1)]}{\ch},
\end{split}
\]
where the first equality follows from \eqref{e:Tpsidrv}. The above is finite by the hypothesis of this section that $j^*$ satisfies \eqref{e:Zui}.

\end{proof}

\begin{theorem}\label{t:eq}
	Suppose that Assumption \ref{a:fpexists} holds. Then, $(H^*,\theta^*)$ is an equilibrium where $H^*$ is given by \eqref{e:defH*} and $\theta^*$ follows \eqref{e:theta*}. The total demand in equilibrium follows \eqref{e:Yeq}.  Moreover, the following properties of equilibrium holds:
	\begin{enumerate}
		\item The conditional law of $V$ is given by $\bbP(V\in dv|\cF^{Y^*}_t)= \rho^*(t,Y^*_t,v) \Pi(dv)$, where $\rho^*$ is as in \eqref{e:rho*}.  In particular, $\frac{Y^*}{\sigma}$ is a standard Brownian motion in its own filtration.
		\item Informed trader's expected profit is given by
		\be \label{e:ivaleq}
		E^{0,v}[W^{\theta^*}_1]=J(0,0)=\Psi^*(v)+ E^{0,v}[\phi^*(\sigma B_1)].
		\ee
		\item Suppose further that one can differentiate inside the expectation\footnote{Recall from \eqref{e:pdeH*} that $H(t,Y^*_t)=\bbE[\frac{d\phi^*}{dy}(Y^*_1)|\cF^{Y^*}_t]$.} so that
		\[
		H_y(t,Y^*_t)=\bbE\Big[\frac{d^2\phi^*}{dy^2}(Y^*_1)\big|\cF^{Y^*}_t\Big].
		\]
		Then, the expected loss of noise traders\footnote{That is, assuming they are not subject to additional transaction costs unlike the informed trader.} is given by
		\be \label{e:noiseloss}
		\sigma^2\bbE\Big[\frac{d^2\phi^*}{dy^2}(Y^*_1)\Big].
		\ee
		\item The insider expects to pay the following amount of penalty in equilibrium:
		\be \label{e:entpen}
		\frac{c}{2}E^{0,v}\big[\int_0^1 (\alpha^*_t)^2dt\big]=\ch H(\bbQ^*||\bbW)=v(\Psi^*)'(v)-\Psi^*(v)-E^{0,v}[\phi^*(Y^*_1)],
		\ee
		where $H(\bbQ^*||\bbW)$ is the relative entropy of $\bbQ^*$ with respect to the Wiener measure, where $\bbQ^*$ is the law of the demand process in equilibrium, as given by \eqref{e:Yeq}, from the point of view of the insider.
	\end{enumerate} 
\end{theorem}
\begin{proof}
	\begin{enumerate}
		\item {\em Market efficiency:} Note that given $\theta^*$, $Y^*$ is the unique strong solution of \eqref{e:Yeq}. Then it follows from Propositions \ref{p:rhodens} and \ref{p:psilb} that $\bbP(V\in dv|\cF^{Y^*}_t)= \rho^*(t,Y^*_t,v) \Pi(dv)$ and \eqref{e:Zsquareint} holds with $\rho$ replaced by $\rho^*$ since  $\rho^*(0,0,\cdot)\equiv 1$ by construction. This in particular  implies that  and $\frac{Y^*}{\sigma}$ is a Brownian motion in its own filtration by considering the optional projection of the drift. Moreover, \[
		H^*(t,Y^*_t)=\bbE[V|\cF^{Y^*}_t]
		\]
		by \eqref{e:defH*}. That is, the market efficiency condition of equilibrium is satisfied. 
		\item {\em Insider optimality and value function:} Suppose $H^*$ is the pricing rule of the market makers, and consider an admissible strategy $d\theta_t=\alpha_t dt$. Then, Ito's formula and the PDE \eqref{e:J} yield
		\be \label{e:Jver}
		\begin{split}
			J(u,Y_u)&= J(0,0) +\int_0^u J_y(t,Y_t)\{\sigma dB_t +\alpha_tdt\}-\int_0^u \frac{J_y(t,Y_t) +v -H^*(t,Y_t)}{2c}dt\\
			&= J(0,0) +\int_0^u J_y(t,Y_t)\sigma dB_t-\int_0^u \frac{\big(J_y(t,Y_t) +v -H^*(t,Y_t)-\alpha c\big)^2}{2c}dt\\
			& +\frac{c}{2}\int_0^u \alpha_t^2dt -\int_0^u (v-H^*(t,Y_t))\alpha_t dt.
		\end{split}
		\ee
		First observe that $J\geq 0$ by  a simple application of Jensen's inequality. Moreover, admissibility condition on $\theta$ implies $\alpha$ and $H^*$ satisfies \eqref{mm:e:theta_cond_2}. Thus, the local martingale in \eqref{e:Jver} is bounded from below by an integrable random variable. That is, $(\int_0^u J_y(t,Y_t) dB_t)$ is a $P^{0,v}$ supermartingale. Using the boundary condition of $J$, one obtains
		\[
		E^{0,v}W^{\theta}_1 \leq J(0,0)-E^{0,v}\int_0^u \frac{\big(J_y(t,Y_t +v) -H^*(t,Y_t)-\alpha c\big)^2}{2c}dt\leq J(0,0).
		\]
		Note that the first inequality becomes  equality if $(\int_0^u J_y(t,Y_t) dB_t)$ is a martingale. Moreover, the second term will disappear when $\theta=\theta^*$. Thus, $\theta^*$ is optimal provided it is admissible and  the stochastic integral above is a martingale. In this case, the expected wealth is given by
		\[
		J(0,0)= J^0(0,0) +\ch \log \rho^*(0,0,v)=E^{0,v}[j^*(\sigma B_1,v)].
		\]	
		which coincides with \eqref{e:ivaleq}.
		\item {\em Optimality of $\theta^*$:} Proposition \ref{p:entropy} yields that 
		\be \label{*admis}
		E^{0,v} \Big[\int_0^1\Big(\frac{\rho_y(s,Y_s,v)}{\rho(s,Y_s,v)}\Big)^2ds\Big]<\infty,
		\ee
		for all $v$ in the support of $V$.  Thus, $\theta^*$ is admissible as soon as $E^{0,v}\int_0^1 H^2(t,Y^*_t)dt<\infty$.  
		
		To this end, note that, since $H(t,Y^*_t)=\bbE[V|\cF^{Y^*}_t]$, $\bbE[H^2(t,Y^*_t)]\leq \bbE[V^2]<\infty$. Thus,  $E^{0,v}\int_0^1 H^2(t,Y^*_t)dt$ is finite for almost all $v$. On the other hand, it follows from \eqref{e:htransform} that
		\[
		E^{0,v}\Big[\int_0^1 H^2(t,Y^*_t)dt\Big]= \frac{\bbE^{\bbW} \Big[\int_0^1 \exp\Big(\frac{v\sigma X_1-\phi^*(\sigma X_1)}{\ch}\Big)H^2(t,\sigma X_t)dt \Big]}{\bbE^{\bbW} \Big[\exp\Big(\frac{v\sigma X_1-\phi^*(\sigma X_1)}{\ch}\Big)\Big]}.
		\]
		Note that the denominator is finite for each $v$ by Lemma \ref{l:rho}, and convex in $v$. Moreover, it is larger than $\exp(-\ch^{-1}\bbE^{\bbW}[\phi^*(\sigma X_1)])>0$. Thus, the numerator is finite for almost all $v$. Since it is also convex in $v$, its finiteness for all $v$ follows. Therefore,   $E^{0,v}\int_0^1 H^2(t,Y^*_t)dt<\infty$ for all $v$.
		
		Finally, it remains to show that $(\int_0^u J_y(t,Y_t) dB_t)$ is a $P^{0,v}$-martingale. However, $J_y(t,y,v) = \ch \frac{\rho_y(t,y,v)}{\rho(t,y,v} + H^*(t,y)-v$. Thus, $(\int_0^u J_y(t,Y_t) dB_t)$ is a $P^{0,v}$-martingale since $E^{0,v}\int_0^1 H^2(t,Y^*_t)dt<\infty$ and \eqref{*admis} holds. 
		\item {\em Noise traders loss:} It follows from integration by parts that the wealth of noise traders is given by
		\[
		\sigma B_1(V-H^*(1,Y_1))-\sigma \int_0^1 B_tdH^*(t,Y^*_t)=\sigma B_1V-\sigma \int_0^1 H^*(t,Y^*_t)dB_t-\sigma^2\int_0^1 H^*_y(t,Y^*_t)dt.
		\]
		The admissibility of $\theta^*$ implies in particular that the stochastic integral with respect to $B$ is a martingale. Moreover, $H^*_y(t,Y^*_t)$ is an $\cF^{Y^*}$-martingale. Thus, taking expectations leads to the following expected wealth:
		\[
		-\sigma^2 \bbE[H^*_y(1,Y^*_1)]=-\sigma^2 \bbE\Big[\frac{d^2\phi^*}{dy^2}(Y^*_1)\Big].
		\]
		\item {\em Expecgted penalty in equilibrium:} This is a direct consequence of Proposition \ref{p:entropy} since 
		\[
		\alpha^*_t=  \sigma^2 \frac{\rho^*_y(t,Y_t,V)}{\rho^*(t,Y_t,V)}.
		\]
	\end{enumerate}
\end{proof}
\begin{remark}
	A simple sufficient condition for differentiability under the expectation in item (3) above is if the second derivative of $\phi^*$ has at most polynomial growth. This would be the case if $V$ has a bounded support. Indeed,  by Theorem \ref{t:fpbd} a fixed point of $T$ exists and has bounded first and second derivatives by Proposition \ref{p:psilb}
\end{remark}
An intresting connection between the expected  loss of the noise traders and the price efficiency of equilibrium is given in the next corollary. Recall that in Back \cite{Back92} the equilibrium price is efficient in the sense that  all the private information of the informed trader is fully disseminated to the market by the end of trading horizon. That is, the conditional variance of $V$ given the market information at time $1$ is zero. This is no longer the case in the present model as the market price does not converge to $V$ as $t\rar 1$. The following defines a measure of {\em price inefficiency} of equilibrium and gives its magnitude.
\begin{corollary}
	Consider the equilibrium of Theorem \ref{t:eq}. Then the {\em price inefficiency} of the equilibrium, denoted by $\delta$, is given by
	\be \label{e:EcVar1}
	\delta:=\bbE[\mbox{Var}(V|\cF^{Y^*}_1)]=\ch\bbE\Big[\frac{d^2\phi^*}{dy^2}(Y^*_1)\Big].
	\ee
	Moreover,
	\be \label{e:cvar}
	\mbox{Var}(V|\cF^{Y^*}_t)=\bbE\Big[\ch \frac{d^2\phi^*}{dy^2}(Y^*_1)+(h^*(Y^*_1))^2|\cF^{Y^*}_t\Big]- (H^*(t,Y^*_t))^2.
	\ee
\end{corollary}
\begin{proof}
	It follows from \eqref{e:Tphidrv2} that $\frac{d^2\phi^*}{dy^2}(y) =\ch^{-1}\mbox{Var}(V|\cF^{Y^*}_1, Y^*_1=y)$, which yields th first claim.
	
	To show the second claim, it suffices to show that $\bbE[V^2|\cF^{Y^*}_t]=\bbE[\ch \frac{d^2\phi^*}{dy^2}(Y^*_1)+(h^*(Y^*_1))^2|\cF^{Y^*}_t]$. However, this follows from the fact that  $\ch \frac{d^2\phi^*}{dy^2}(Y^*_1) =\mbox{Var}(V|\cF^{Y^*}_1)$ and $\bbE[V|\cF^{Y^*}_1]= h^*(Y^*_1)$.
\end{proof}
\begin{remark}
	Comparing the above to the loss  of noise traders from Theorem \ref{t:eq} reveals the interesting observation that, independent of the distribution of the fundamental value of the asset, the expected loss of noise traders equals $c \delta$, where $\delta$ is the price inefficiency. 
\end{remark}
\subsection{Gaussian case}
Although Theorem \ref{t:fpbd} establishes the existence of a fixed point for $V$ with a bounded support, there is a fixed point when $V$ has a Gaussian distribution. This subsection studies the existence of a  fixed point and the resulting equilibrium under the following assumption:
\begin{assumption}
	\label{a:gaussian}
	The distribution of $V$ is Gaussian, i.e. $V \sim N(\mu,\gamma^2)$.
\end{assumption}
In the Kyle model with Gaussian fundamental value but no transaction costs, the equilibrium price is an affine function of the demand $Y$. Thus, one expects $H^*(t,y)=\lambda y +\mu$ for some $\lambda>0$. The conjecture that $H^*(t,0)=\mu$ is a simple consequence of the rationality of the pricing rule that requires $\mu=\bbE[V]=S_0=H^*(0,0).$ Combined with the normalisation that $\phi^*(0)=0$, one expects that $\phi^*(y)=\frac{\lambda y^2}{2}+\mu y$. This leads to the following:
\begin{theorem} \label{t:fpG}
	Suppose that Assumpton \ref{a:gaussian} holds. Then, $T_0\phi^*=\phi^*$ with $\phi^*(y)=\frac{\lambda^* y^2}{2}+\mu y$ for 
	\be \label{e:lambdaeq}
\lambda^*=\frac{-c +\sqrt{c^2 + 4 \frac{\gamma^2}{\sigma^2}}}{2},
	\ee
	which is the positive root of $\sigma^2 \lambda^2 + \ch \lambda -\gamma^2=0$.
 Moreover, 
 \be \label{e:psigaus}
 \Psi^*(v)=\ch \log T_1\phi^*(v)=\frac{\ch}{2} \log\frac{c}{c+\lambda^*}+\frac{(\mu-v)^2}{2(c+\lambda^*)}.
 \ee
\end{theorem}
\begin{proof}
By completing the squares, it is easily seen that
\[
T_1\phi^*(y)=\frac{\Sigma}{\sigma}\exp\Big(\frac{\Sigma^2(\mu-v)^2}{2 \ch^2}\Big),
\]
where $\frac{1}{\Sigma^2}= \frac{1}{\sigma^2}+\frac{\lambda}{\ch}=\frac{c+\lambda}{\ch}$. That is,
\[
T_1\phi^*(y)=\frac{\Sigma}{\sigma}\exp\Big(\frac{(\mu-v)^2}{2 \ch(c+\lambda)}\Big).
\] 
Thus, to have $T_0\phi^*=\phi^*$, one must have 
\[
\begin{split}
\exp\big(\frac{\lambda y^2+ 2\mu y}{2\ch}\big)&=\frac{\sigma}{\Sigma\gamma}\int_{\bbR}\frac{1}{\sqrt{2 \pi }}\exp\Big(\frac{yv}{\ch}-\frac{(\mu-v)^2}{2 \ch(c+\lambda)}-\frac{(\mu-v)^2}{2 \gamma^2}\Big)dv\\
&=\frac{\sigma\Sigma_0}{\Sigma\gamma}\int_{\bbR}\frac{1}{\sqrt{2 \pi \Sigma_0^2}}\exp\Big(\frac{yv}{\ch}-\frac{(\mu-v)^2}{2\Sigma_0^2)}\Big)dv\\
&=\frac{\sigma\Sigma_0}{\Sigma\gamma}\exp\Big(\frac{y\mu}{\ch}+\frac{y^2 \Sigma_0^2}{2 \ch^2}\Big),
\end{split}
\]
where $\frac{1}{\Sigma_0^2}=\frac{1}{\ch(c+\lambda)}+\frac{1}{\gamma^2}=\frac{\gamma^2+\ch(c+\lambda)}{\gamma^2\ch(c+\lambda)}$. 

Thus, $\lambda$ must solve
\be\label{e:lambda1}
\lambda =\frac{\Sigma_0^2}{\ch}=\frac{\gamma^2(c+\lambda)}{\gamma^2+\ch(c+\lambda)}.
\ee
Note that one must also have
\be \label{e:lambda2}
1=\frac{\sigma^2 \Sigma_0^2}{\Sigma^2\gamma^2}=\frac{\sigma^2(c+\lambda)}{\gamma^2}\lambda.
\ee
That is, $\lambda$ must seemingly  satisfy two constraints at the same time. However, \eqref{e:lambda1} is equivalent to 
\[
\lambda c\sigma^2(c+\lambda)=c\gamma^2.
\]
Dividing both sides by $c$ yields \eqref{e:lambda2}. Thus, $\lambda$ is the unique positive solution of 
\[
\lambda^2\sigma^2 + \lambda c\sigma^2 -\gamma^2=0,
\]
which is given by $\lambda^*$.

Consequently,
\[
\Psi^*(v)=\frac{\ch}{2} \log \frac{\Sigma^2}{\sigma^2} +\frac{(\mu-v)^2}{2(c+\lambda^*)}=\frac{\ch}{2} \log\frac{c}{c+\lambda^*}+\frac{(\mu-v)^2}{2(c+\lambda^*)}.
\]
\end{proof}
The formula for $\lambda^*$ suggests a parametrisation for the cost coefficient $c$ that will be useful when considering the equilibrium with Gaussian fundamental value. The easy proof of the following corollary is left to the reader.
\begin{corollary} \label{c:parametr}
	Let $c=\kappa \frac{\gamma}{\sigma}$ for some $\kappa>0$. Then, $\lambda^*= \Lambda(\kappa)\frac{\gamma}{\sigma}$, where $\Lambda:(0,\infty)\to (0,1)$ is a decreasing function given by
	\[
	\Lambda(\kappa)=\frac{\sqrt{\kappa^2+4}-\kappa}{2}.
	\]
	In particular, $\Lambda(0+)=1=1-\Lambda(\infty)$.
\end{corollary}
Note that  the fixed point of $T_0$ given by the above result satisfies all the conditions of Theorem \ref{t:eq}. Thus, an equilibrium exists and has the following dynamics as a straightforward corollary. To understand the dependency of the equilibrium on the amount of adverse selection $\frac{\gamma}{\sigma}$, the above parametrisation will be employed.
\begin{corollary}\label{c:Gequilibrium}
	Suppose that Assumption \ref{a:gaussian} holds and $c=\kappa \frac{\gamma}{\sigma}$ for some $\kappa>0$. Then $(h^*,\theta^*)$ is an equilibrium, where $h^*(y)=\lambda^*y +\mu$ with $\lambda^*$ from \eqref{e:lambdaeq}, and 
	\be \label{e:alph*G}
	d\theta^*_t= \Lambda(\kappa)\sigma\frac{\frac{V-\mu}{\gamma} -\frac{\Lambda(\kappa)Y^*_t}{\sigma}}{1- t\Lambda^2(\kappa)}dt,
	\ee
	where $\Lambda$ is the function defined in Corollary \ref{c:parametr}. Moreover, 
	\begin{enumerate}
		\item The conditional law of $V$ given $\cF^{Y^*_1}$ is Gaussian with mean $\lambda^* Y^*_t+\mu$ and variance $\gamma^2(1-t\Lambda^2(\kappa))$. Thus,
		\[
		\rho^*(t,y,v)=\frac{1}{\sqrt{1-t\Lambda^2(\kappa)}}\exp\Big(-\frac{1}{2\gamma^2}\Big(\frac{(v-\lambda^*y-\mu)^2}{1-t\Lambda^2(\kappa)}-v^2\Big)\Big)
		\]
		\item Equilibrium demand $Y^*$ follows
		\[
		dY^*_t=\sigma dB_t + \Lambda(\kappa)\sigma\frac{\frac{V-\mu}{\gamma} -\frac{\Lambda(\kappa)Y^*_t}{\sigma}}{1- t\Lambda^2(\kappa)}dt.
		\]
		\item Informed trader's expected profit is given by
		\be \label{e:ivalG}
		E^{0,v}[W^{\theta^*}_1]=\frac{\ch}{2} \log\frac{c}{c+\lambda^*}+\frac{(\mu-v)^2}{2(c+\lambda^*)}+\frac{\lambda^*\sigma^2}{2}.
		\ee
		\item The expected loss of noise traders is given by $\lambda^*\sigma^2=\Lambda(\kappa)\gamma\sigma$.
		\item The price inefficiency  is given by $\delta=\ch \lambda^*=\kappa \Lambda(\kappa)\gamma^2$.  Moreover,
		\[
			\mbox{Var}(V|\cF^{Y^*}_t)=\kappa \Lambda(\kappa)\gamma^2 +\Lambda^2(\kappa)\gamma^2 (1-t)=\gamma^2 (1-t)+t\kappa \Lambda(\kappa)\gamma^2.
		\]
	\end{enumerate}
	
\end{corollary}
\begin{proof}
	The only claim that is not explicitly given by Theorems \ref{t:eq} and \ref{t:fpG} is the conditional distribution of $V$. Note that the conditonal distribution at time $1$ is given by $\exp(-j^*(y,v)/c)\Pi(dv)$, where $j^*$ is given by \eqref{e:j*}.  Straightforward computations show that this distribution is Gaussian with density proportional to 
	\[
	\exp\Big(-(\mu-v)^2\frac{\gamma^2+\ch(c+\lambda^*)}{2\gamma^2\ch(c+\lambda^*)}+\frac{vy}{\ch}\Big).
	\]
	Since $\frac{\gamma^2(c+\lambda^*)}{\gamma^2+\ch(c+\lambda^*)}=\lambda^*\ch$ by \eqref{e:lambda1}, it follows that the conditional distribution at time $1$ is Gaussian with mean $\lambda^*Y^*_1 +\mu$ and variance $\ch \lambda^*$.  Thus, the conditional density at time $t$ given $Y^*_t=y$ equals
	\[
\int_{\bbR}\frac{1}{\sqrt{2\pi\sigma^2(1-t)}}\exp\Big(-\frac{(y-z)^2}{2 \sigma^2(1-t)}\Big)\frac{1}{\sqrt{2 \pi \ch \lambda^*}}\exp\Big(-\frac{(v-\mu-\lambda^* z)^2}{2 \ch \lambda^*}\Big)dzdv=\rho^*(t,y,v)\Pi(dv).
\]
This will be a Gaussian density as well. Indeed, by considering its characteristic function
\[
\begin{split}
\int_{\bbR}e^{irv}\rho^*(t,y,v)\Pi(dv)&=e^{-\frac{r^2\ch\lambda^*}{2}}\int_{\bbR}\frac{e^{ir(\mu+\lambda^*z)}}{\sqrt{2\pi\sigma^2(1-t)}}\exp\Big(-\frac{(y-z)^2}{2 \sigma^2(1-t)}\Big)dz\\
&=e^{-\frac{r^2}{2}(\ch\lambda^* +(\lambda^*)^2 \sigma^2(1-t))+ir(\mu+\lambda^*y)}.
\end{split}
\]
Therefore,  the conditional distribution is Gaussian with mean $\mu+\lambda^*Y^*_t$ and variance $\ch \lambda^*+ ( \lambda^*)^2\sigma^2(1-t)$.  Using the quadratic equation that $\lambda^*$ solves from Theorem \ref{t:fpG}, the variance in fact equals $\gamma^2- t(\lambda^*\sigma)^2$. Using the density of normal, this readily yields the equilibrium $\theta^*$ using Corollary \ref{c:parametr}.
\end{proof}

The above shows that the insider buys when the market valuation is lower than the fundamental value and sells otherwise as in the Kyle model without penalties. Thus, from the point of view of the insider, prices still mean revert around insider's valuation.  However, the speed of mean reversion as measured by 
\be \label{e:speedins}
r(t):=\frac{\Lambda(\kappa)}{1-t \Lambda^2(\kappa)}
\ee
is considerably smaller than what it would be in  the model without penalties, where the corresponding speed equals $\frac{1}{1-t}$. Thus, the insider trades less aggressively when there is a legal risk. This in particular results in prices not fully revealing the insider's private information at the end of the trading horizon. 

Recall that {\em the Kyle's lambda} equals $\frac{\gamma}{\sigma}$ when there are no penalties. Thus, 
\be \label{e:gainliq}
L(\kappa):=\frac{1}{\Lambda(\kappa)}-1=\frac{1-\Lambda(\kappa)}{\Lambda(\kappa)}
\ee
can be used to  represent the gain in liquidity when the regulators impose a quadratic penalty with rate given by $\kappa \frac{\gamma}{\sigma}$. This gain is clearly decreasing in $\Lambda$, which is a decreasing function of $\kappa$. The market thus becomes perfectly liquid in the limit but it will also be completely inefficient.

Another interesting consequence of the reparametrisation manifests itself in the representation of insider's ex-ante profit. Note that $\Lambda(\kappa)$ is the positive solution of 
\be \label{e:Lambdaeqn}
\Lambda^2(\kappa) +\Lambda(\kappa) \kappa=1.
\ee
\begin{corollary}
	Consider the equilibrium in Corollary \ref{c:Gequilibrium}. Then, the following statements are valid:
	\begin{enumerate}
		\item The ex-ante profit of the insider equals
		\[
	\cW(\kappa):=	\bbE[W^{\theta^*}_1]=\frac{\gamma \sigma}{2}\big(\kappa \log (\kappa \Lambda(\kappa))+2\Lambda(\kappa) \big).
		\]
		Moreover, it is convex and decreasing in $\kappa$ and vanishes in the limit as $\kappa \rar \infty$.
		\item Derivative of $\Lambda(\kappa)$ is given by
		\[
		 \Lambda'(\kappa)=-\frac{\Lambda^2(\kappa)}{1+\Lambda^2(\kappa)}.
		 \]
		\item The total welfare that measures the total expected profit of all traders is given by
		\be \label{e:welfare}
		\frac{\gamma \sigma}{2}\kappa \log (\kappa \Lambda(\kappa))= \frac{\gamma \sigma}{2}\kappa \log (1- \Lambda^2(\kappa)).
		\ee
	\end{enumerate}
\end{corollary}
\begin{proof}
	The first and the last claims follow from Corollary \ref{c:Gequilibrium} upon noticing that $(c+\lambda^*)\frac{\sigma}{\gamma}=\kappa+\Lambda(\kappa)= \frac{1}{\Lambda(\kappa)}$ as well as $\kappa \Lambda(\kappa)=1-\Lambda^2(\kappa)$ in view of \eqref{e:Lambdaeqn}.
	
	Differentiating $\cW$ with respect to $\kappa$ and noting $1-\Lambda^2(\kappa)=\Lambda(\kappa)\kappa$ yield
	\be \label{e:Wder}
	 \cW'(\kappa)=\frac{\gamma\sigma}{2}\log(1-\Lambda^2(\kappa)), \quad \cW''(\kappa)=-\gamma\sigma \Lambda'(\kappa)\kappa^{-1}.
	 \ee
	 Moreover, \eqref{e:Lambdaeqn} yields
	 \[
	 \Lambda'(\kappa)= -\frac{\Lambda(\kappa)}{2\Lambda +\kappa}=-\frac{\Lambda(\kappa)}{\Lambda(\kappa) +\frac{1}{\Lambda(\kappa)}}=-\frac{\Lambda^2(\kappa)}{1+\Lambda^2(\kappa)}<0.
	 \]
	 This completes the proof the claims on derivatives. To show that the limiting wealth is $0$, first recall from Corollary \ref{c:parametr} that $\Lambda(\infty)=0$.  On the other hand,
	 \[
	 \kappa \log (\kappa \Lambda(\kappa))=\frac{1-\Lambda^2(\kappa)}{\Lambda(\kappa)}\log(1-\Lambda^2(\kappa)).
	 \]
	 Thus, the proof will be complete if $\lim_{x\rar 0}\frac{\log(1-x^2)}{x}=0$, which follows by a quick application of the L'Hospital rule.
\end{proof}
As $\Lambda(K) \in (0,1)$, the above shows that the welfare is negative. This is no surprise since the market makers are competitive and, therefore, the total welfare equals in magnitude to the expected penalty to the insider.

Similarly, one also expect that the insider's total wealth vanishes as rate of penalties increases without bound. The following result shows that the expected penalty in equilibrium is bounded in $\kappa$. Thus, the vanishing wealth in the limit is not a consequence of large expected costs  but insider trading very little in the presence of legal risk. This is consistent with the representation of the speed of trading via \eqref{e:speedins}. Indeed, $r$ is a decreasing function of $\kappa$ and vanishes as $\kappa \rar \infty$. 

The expected penalty is also not monotone as can be deduced from the following result or seen in Figure \ref{f:penalty}.
\begin{corollary}
Consider the equilibrium in Corollary \ref{c:Gequilibrium}. The insider's expected penalties  in equilibrium equals
\[
P(\kappa)= -\frac{\gamma \sigma}{2}\kappa \log (\kappa \Lambda(\kappa)).
\]
There exists a $\kappa_0$ such that $P$ is concave on $(0,\kappa_0)$ and convex  otherwise.  In particular, it is bounded by $\frac{\gamma \sigma}{2}$. Moreover, $P(0+)=P(\infty)=0$. 
\end{corollary}
\begin{proof} 
	Note that for $x\in (0,1)$,
	\[
	\log(1-x)\geq -\frac{x}{1-x}
	\]
	in view of the mean value theorem.  Thus,
	\[
\frac{1-\Lambda^2(\kappa)}{\Lambda(\kappa)}\log(1-\Lambda^2(\kappa))\geq -\Lambda(\kappa)\geq -1.
	\]
	This shows the bound on $P$. 
	
	Moreover, \eqref{e:Wder} implies
	\[
	P''(\kappa)=\gamma\sigma  (\Lambda'(\kappa)\kappa^{-1}+2 \Lambda''(\kappa))=\gamma\sigma\Big(-\frac{\kappa^{-1}\Lambda^2(\kappa)}{1+\Lambda^2(\kappa)}+4\frac{\Lambda^3(K)}{(1+\Lambda^2(\kappa))^3}\Big),
	\]
	which is positive if and only if
	\[
	\frac{(1+\Lambda^2(\kappa))^2}{\kappa}\leq 4.
	\]
	Note that the left side is  decreasing from $\infty$ to $0$ as $\kappa$ increases to $\infty$.  Thus, there exists a $\kappa_0$ such that $P$ is concave on $(0,\kappa_0)$ and convex otherwise. 
	
	The limiting values at $0$ and $\infty$ have already been observed in computations leading to the respective limits for $\cW$ in the proof of the previous corollary. 
	\end{proof}
	\begin{remark}
		In view of the relationship between relative entropy and expected penalties via \eqref{e:entpen}, the `average' entropy, i.e. $H(\bbQ^*||\bbW)$ integrated with respect to the law of $V$, equals $-\frac{1}{2}\log (\kappa \Lambda(\kappa))=-\frac{1}{2}\log (1-\Lambda^2(\kappa))$. This is increasing as $\kappa\rar 0$, i.e. smaller quadratic penalties are imposed by the regulators.  The limiting value is infinite, which coincides with the relative entropy of the Kyle equilibrium without penalties. Note that when there are no penalties, $Y^*/\sigma$ is a Brownian bridge:
		\[
dY^*_t=\sigma dB_t + \sigma\frac{\frac{V-\mu}{\gamma} -\frac{Y^*_t}{\sigma}}{1- t}dt.
\]
The law of $Y^*$ is then singular with respect to the Wiener measure. 
	\end{remark}
\begin{figure}[h]
	\begin{center}
		\scalebox{0.3}{\includegraphics{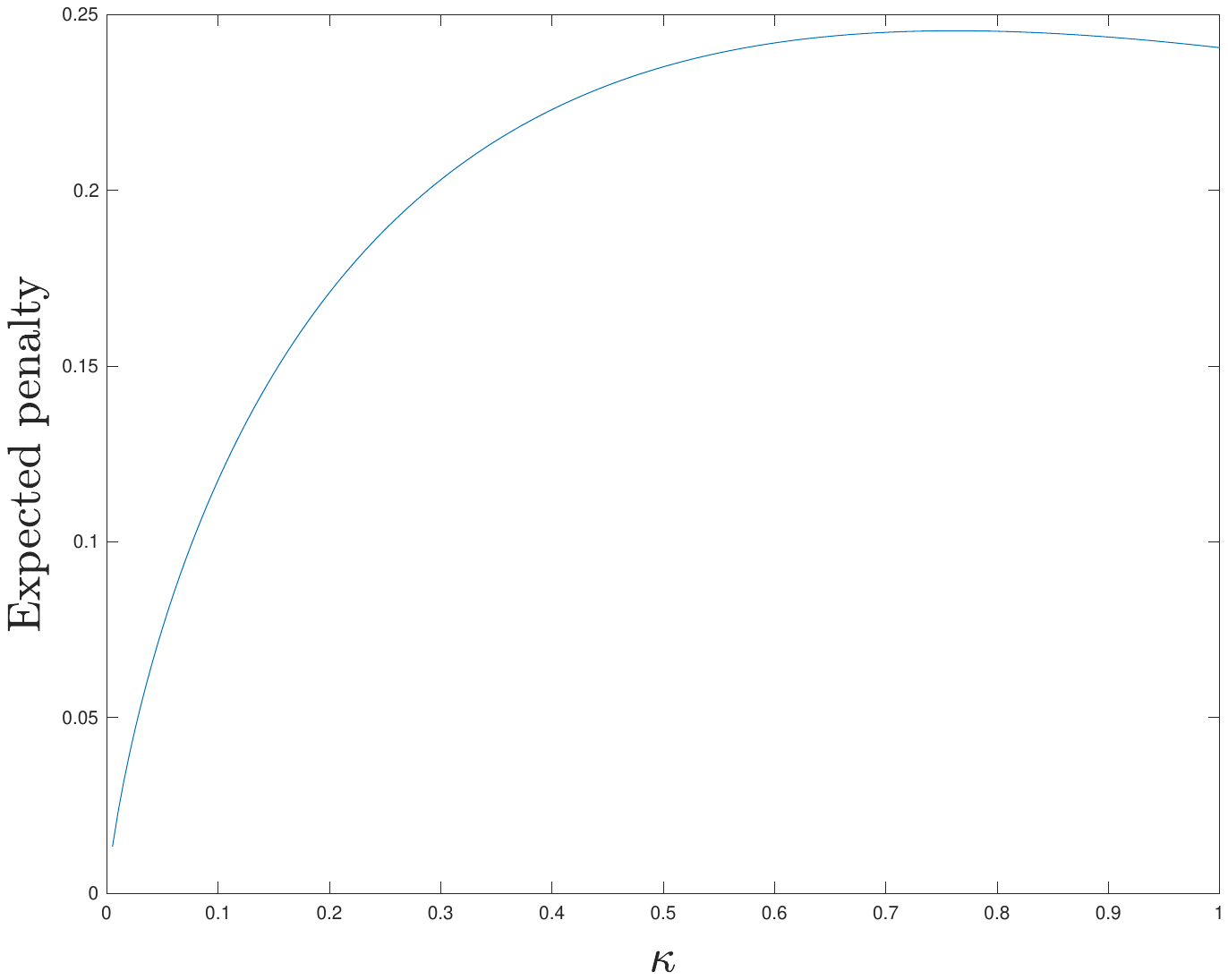}}
	\end{center}
	\caption{Expected penalty paid by the insider in equilibrium as a function of $\kappa$. The values on the $y$-axis are normalised by $\gamma\sigma$. }
	\label{f:penalty}
\end{figure}
Although the expected penalties (hence the total welfare) are non-monotone, the loss of the noise traders and price efficiency decrease as the rate of penalties, i.e. $\kappa$, increase. These can be verified by direct differentiation of the relevant expressions and are illustrated in Figure \ref{f:noise}.
\begin{figure}[h]
	\begin{center}
		$
		\ba{cc}
		\scalebox{0.26}{\includegraphics{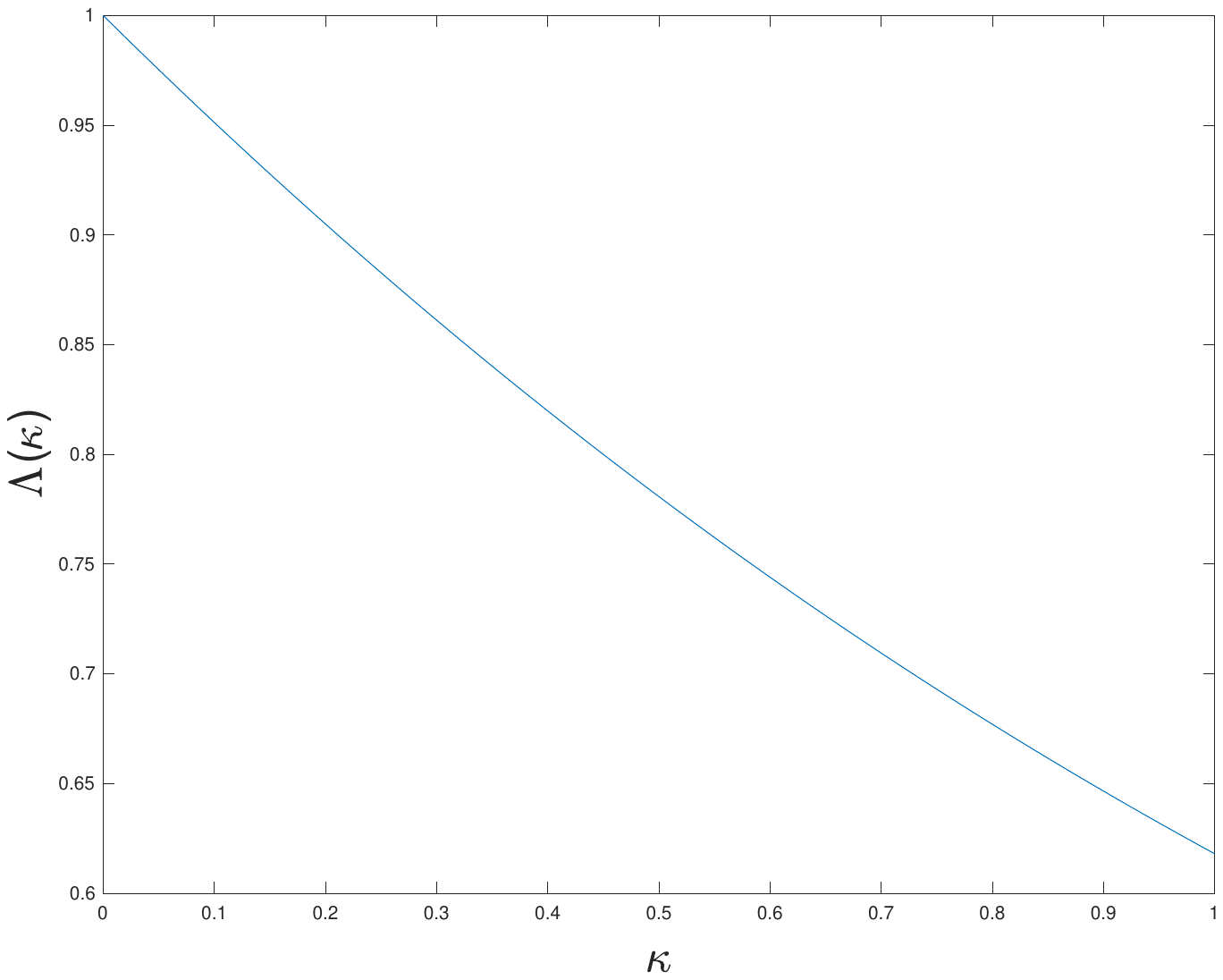}} & \scalebox{0.26}{\includegraphics{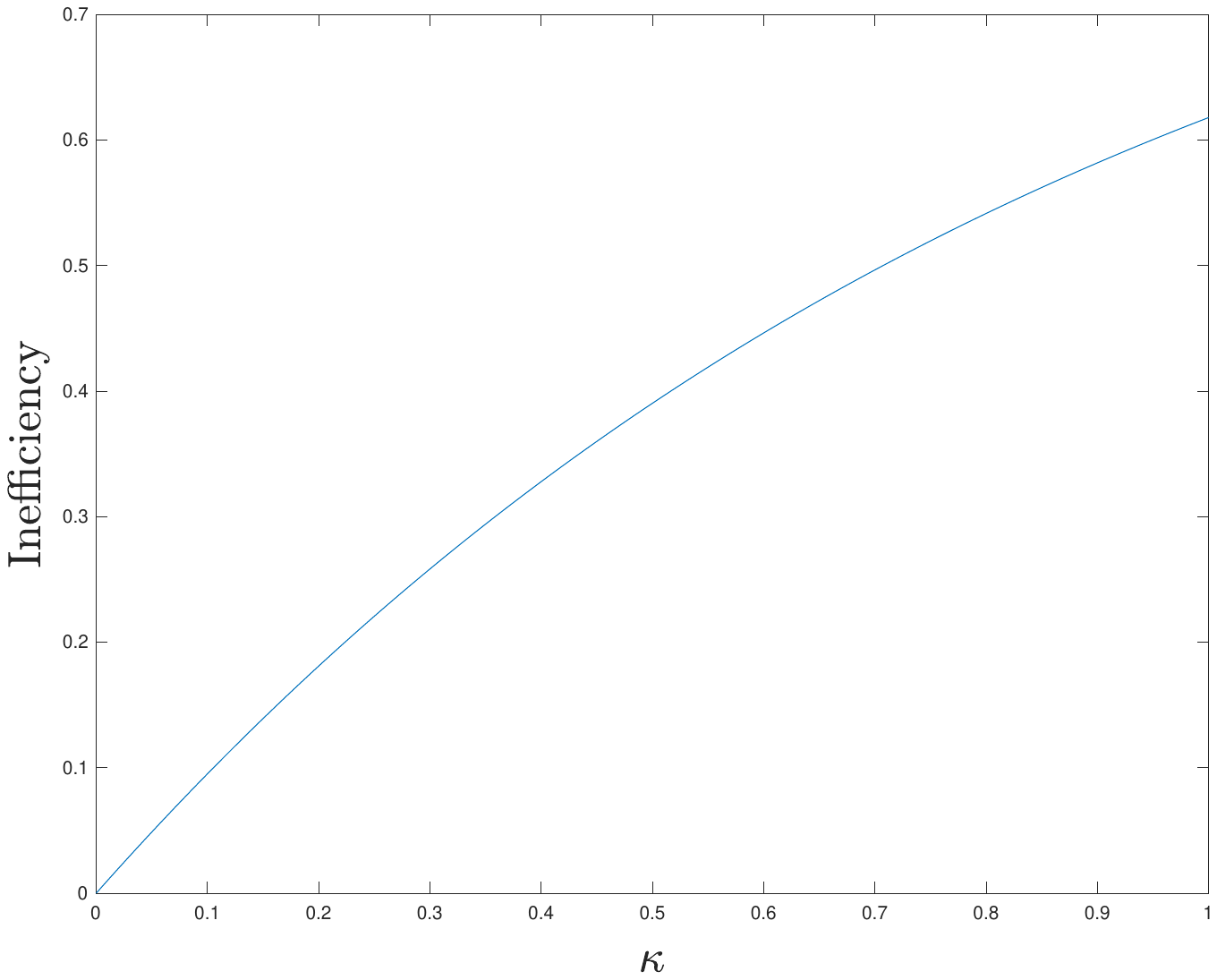}}
		\ea
		$
	\end{center}
	\caption{The left pane shows the expected loss of noise traders normalised by $\gamma\sigma$ whereas the right one is the price inefficiency of equilibrium normalised by $\gamma^2$. }
	\label{f:noise}
\end{figure} 

\section{Conclusion} \label{s:c}
When there are quadratic penalties, the equilibrium in the continuous time version of the Kyle model changes drastically. The insider no longer finds optimal to bring the market prices to her own valuation at the end ot trading. The noise traders lose less on average  as the potential penalties on the insider increase; however,  the total welfare, i.e the sum of insider and noise profit, is not monotone in the rate of penalties when the fundamental value is normally distributed.  Interestingly, the maximum  expected penalty that is accrued by the insider occurs when the rate of penalties is sufficiently small.

The solution of the equilibrium reveals an interesting connection between a class of quadratic BSDEs and h-transforms. The interesting extension to equilibrium with general convex costs are left for future research. Such an extension will be valuable to regulators searching for the optimal penalty structure to control illegal insider trading as discussed by \cite{CCDG} in a one-period setting. 
 \bibliographystyle{plain}
\bibliography{../../ref}
\end{document}